\documentclass[12pt]{amsart}

\usepackage{color}
\usepackage[all]{xy}
\usepackage{mathrsfs}
\usepackage{amsmath, amssymb}
\usepackage{graphics}
\usepackage{graphicx}  
\usepackage{geometry}
\usepackage[cal=boondox,scr=boondoxo]{mathalfa}
\DeclareFontFamily{U}{mathc}{}
\DeclareFontShape{U}{mathc}{m}{it}
{<->s*[1.03] mathc10}{}
\DeclareMathAlphabet{\mathscr}{U}{mathc}{m}{it}

\newtheorem{theo}{Theorem}[section]

\newtheorem{cor}[theo]{Corollary}

\newtheorem{lemma}[theo]{Lemma}
\newtheorem{prop}[theo]{Proposition}
\newtheorem{defi}[theo]{Definition}
\newtheorem{rem}[theo]{Remark}
\newtheorem{example}[theo]{Example}

\newtheorem{nota}{Notation}

\newtheorem*{main}{Main Theorem}

\newcommand{\End}{\operatorname{End}\nolimits}
\newcommand{\Hom}{\operatorname{Hom}\nolimits}
\DeclareMathOperator{\sHom}{\operatorname{\mathscr{Hom}}\nolimits}
\renewcommand{\Im}{\operatorname{Im}\nolimits}
\newcommand{\Ker}{\operatorname{Ker}\nolimits}
\newcommand{\Coker}{\operatorname{Coker}\nolimits}
\newcommand{\id}{\operatorname{id}\nolimits}
\newcommand{\Ext}{\operatorname{Ext}\nolimits}
\newcommand{\hExt}{\operatorname{\widehat{Ext}}\nolimits}
\newcommand{\Tor}{\operatorname{Tor}\nolimits}
\newcommand{\hTor}{\operatorname{\widehat{Tor}}\nolimits}
\newcommand{\hotimes}{\operatorname{\widehat{\otimes}}\nolimits}

\newcommand{\Z}{\operatorname{\mathbb{Z}}\nolimits}

\newcommand{\N}{\operatorname{\mathbb{N}}\nolimits}
\renewcommand{\P}{\operatorname{\mathbf{P}}\nolimits}
\newcommand{\I}{\operatorname{\mathbf{I}}\nolimits}
\newcommand{\T}{\operatorname{\mathbf{T}}\nolimits}
\newcommand{\D}{\operatorname{\mathbb{D}}\nolimits}
\newcommand{\Aut}{\operatorname{Aut}\nolimits}

\newcommand{\Ae}{A^\textrm{e}}

\newcommand{\Hh}{\operatorname{H}\nolimits}
\newcommand{\hHh}{\operatorname{\widehat{H}}\nolimits}
\newcommand{\HH}{\operatorname{HH}\nolimits}
\newcommand{\hHH}{\operatorname{\widehat{HH}}\nolimits}

\numberwithin{equation}{section}

\begin{document}
\title[Algebraic structure on Tate-Hochschild cohomology]{Algebraic structure on Tate-Hochschild cohomology of a Frobenius algebra}

\author{Satoshi Usui}
\address{Graduate School of Mathematics, 
Tokyo University of Science, 
1-3 Kagurazaka, Shinjuku-ku, Tokyo 162-8601 JAPAN}
\curraddr{}
\email{1119702@ed.tus.ac.jp}
\thanks{}



\date{}

\dedicatory{}

\begin{abstract}
We study cup product and cap product in Tate-Hochschild theory 
for a finite dimensional Frobenius algebra.
We show that Tate-Hochschild cohomology ring equipped with cup product is 
isomorphic to singular Hochschild cohomology ring introduced by Wang. 

An application of cap product occurs in Tate-Hochschild duality; 
as in Tate (co)homology of a finite group, 
the cap product with the fundamental class of 
a finite dimensional Frobenius algebra provides
 certain duality result between Tate-Hochschild cohomology and homology groups.
 
Moreover, we characterize minimal complete resolutions 
over a finite dimensional self-injective algebra 
by means of the notion of minimal complexes 
introduced by Avramov and Martsinkovsky.
\end{abstract}

\maketitle

\tableofcontents

%
%
\section*{Introduction}
For an associative algebra $A$ which is projective 
over a ground commutative ring $k$, 
\textit{Hochschild cohomology groups} $\Hh^{*}(A, M)$ of $A$ 
with coefficients in an $A$-bimodule $M$
are defined by the cohomology groups of  
the cochain complex $\sHom_{A \otimes_{k} A^{\circ}}(\P, M)$,
where  $\P$ is arbitrary projective resolution of $A$ as an $A$-bimoudle.
Recall that $\Hh^{i}(A, M)$ is isomorphic to the space of morphisms from $A$ to $\boldsymbol{\Sigma}^{i} M$ 
in the bounded derived category $\mathcal{D}^{{\rm b}}(A \otimes_{k} A^{\circ})$ of $A$-bimodules for $i \geq 0$, 
where $\boldsymbol{\Sigma}$ is the shift functor.
There is an operator, the so-called \textit{cup product} $\smile$,
$\Hh^{*}(A, M) \otimes \Hh^{*}(A, N) 
\rightarrow 
\Hh^{*}(A, M \otimes_{A} N)
$
for $A$-bimodules $M$ and $N$.
It is defined, on the level of complexes, by means of a \textit{diagonal approximation} associated with a given projective resolution of $A$. 
In particular, the cup product does not depend on the choice of a diagonal approximation and a projective resolution.
If $M = A$ and $N= A$, then Hochschild cohomology 
\[
\Hh^{\bullet}(A, A) := \bigoplus_{i \geq 0} \Hh^{i}(A, A)
\]
equipped with the cup product forms a graded algebra. 
In \cite{Gers63}, Gerstenhaber showed that 
the cup product induced by a certain diagonal approximation
obtained from the bar resolution of an associative algebra 
is graded commutative, 
that is, the Hochschild cohomology ring based on the bar resolution 
is graded commutative.
In conclusion, the result of him implies that 
the Hochschild cohomology ring via any projective resolution 
of an associative algebra is a graded commutative algebra.

Inspired by the Buchweitz's result 
on Tate cohomology of Iwanaga-Gorenstein algebras (\cite{Buch86}),
Wang \cite{Wang18} introduced \textit{singular Hochschild cochain complex}
$C_{\rm sg}(A, A)$ of an associative algebra $A$ over a field and proved that 
the singular Hochschild cohomology group $\HH_{\rm sg}^{i}(A, A)$ of $A$ is isomorphic to
the space of morphisms from $A$ to $\boldsymbol{\Sigma}^{i} A$ in the singularity category 
$\mathcal{D}_{{\rm sg}}(A \otimes_{k} A^{\circ})$ of $A \otimes_{k} A^{\circ}$ for any $i \in \Z$. 
Furthermore, he discovered a differential graded associative and unital product on $C_{\rm sg}(A, A)$ 
such that singular Hochschild cohomology
\[
\HH_{\rm sg}^{\bullet}(A, A) 
:= 
\bigoplus_{i \in \Z} \HH_{\rm sg}^{i}(A, A)
\]
equipped with the induced product 
is a graded commutative algebra.

In the case that an algebra $A$ is a finite dimensional Frobenius algebra, 
 the singular Hochschild cohomology groups of $A$ coincide with 
the cohomology groups based on a complete resolution of $A$. 
They are called \textit{Tate-Hochschild cohomology groups} of $A$ and denoted 
by $\hHh^{*}(A, A)$.
Therefore, Tate-Hochschild cohomology
\[
\hHh^{\bullet}(A, A) := \bigoplus_{i \in \Z} \hHh^{i}(A, A)
\]
becomes a graded commutative algebra
whose structure depends on the singular Hochschild cohomology ring structure.
On the other hand, Sanada \cite{Sanada92} constructed one (complete) diagonal approximation 
associated with the complete bar resolution of a finite dimensional Frobenius algebra.
In particular, the cup product induced by his diagonal approximation makes the Tate-Hochschild cohomology ring 
into a graded commutative algebra. 
These results motivate us to ask the following questions:
\begin{itemize}
  \setlength{\parskip}{2mm} 
    \item[(1)]
    Is the Tate-Hochschild cohomology ring given by Sanada isomorphic to
    the singular Hochschild cohomology ring introduced by Wang?
    \item[(2)]
    Is there the theory of cup product in Tate-Hochschild theory of any finite dimensional Frobenius algebra as in Hochschild theory?
\end{itemize}
Let us remark on the second question: Nguyen \cite{Nguy13} has already developed  
the theory of cup product on Tate-Hochschild cohomology
of a finite dimensional Hopf algebra.

In this paper, along the same lines as Brown \cite[Chapter VI, Section 5]{Brown94}, we will develop the theory of cup product in Tate-Hochschild theory of a finite dimensional Frobenius algebra.
More precisely, we will show that the existence of a diagonal approximation 
for arbitrary complete resolution of a given finite dimensional Frobenius algebra and 
that all diagonal approximations define exactly one cup product up to isomorphism (see Section \ref{sebsec:1}).
 That is the answer to the question (2).
Moreover, we will prove the following our main result, which is the answer to the question (1):
\begin{main}[{Theorem \ref{theo:7}}]
Let $A$ be a finite dimensional Frobenius algebra over a field. Then there exists an isomorphism
\[
\hHh^{\bullet}(A, A) \cong \HH_{{\rm sg}}^{\bullet}(A, A)
\] 
of graded commutative algebras. 
\end{main}
\noindent
We also deal with cap product in Tate-Hochschild theory and 
show that the cap product with the \textit{fundamental class} 
of a finite dimensional Frobenius algebra
gives certain duality result between Tate-Hochschild cohomology and homology groups.
These results allow us to prove that the cup product on Tate-Hochschild cohomology contains 
not only the cup product, but also the cap product on Hochschild (co)homology.

Moreover, we provide a characterization of minimal complete resolutions of finitely generated modules
over a finite dimensional self-injective algebra in the sense of Avramov and Martsinkovsky\ \cite{AvraMart02}. 
More concretely, we will show that any minimal complete resolution of a finitely generated module consists of its minimal projective resolution
and its (($-1$)-shifted) minimal injective resolution.

This paper is organized as follows: 
%
%
in Section $\ref{sec:1}$, we recall the basic notions related to Hochschild (co)homology groups 
and the cup product and the cap product on them.
%
%
Section $\ref{sec:2}$ is devoted to recalling the definitions of Tate and Tate-Hochschild (co)homology groups and to characterizing minimal complete resolutions over a finite dimensional self-injective algebra in terms of projective resolutions and injective resolutions.
%
%
Section $\ref{sec:3}$ contains our main theorem. Before proving it, 
we will define not only cup product, but also cap product
by using a diagonal approximation, and we prove that these operators coincide with composition products.
%
%
In Section $\ref{sec:4}$, we will show that the cap product induces duality 
between Tate-Hochschild cohomology and homology groups.
Using this result, we will prove that the cup product on Tate-Hochschild cohomology extends 
the cup product and the cap product on Hochschild (co)homology.

Throughout this paper, an algebra means an associative and unital algebra 
over a commutative ring, and all modules are left modules unless otherwise stated. 
The ground ring is taken to be a field when we assume a given algebra to be finite dimensional.
For a $k$-algebra $A$, we denote $D(-)= \Hom_{k}(-, k)$ and $(-)^{\vee} := \Hom_{A}(-, A)$.
Finally, $\otimes$ is an abbreviation for $\otimes_{k}$.

%
%
\section{Preliminaries} \label{sec:1}
Following Brown \cite{Brown94}, we briefly recall some basic notions related to Hochschild (co)homology groups and the cup product and the cap product on them.

%
%
\subsection{Hom complexes and tensor products of complexes}
A \textit{chain complex} $C = (C, d^{C})$ over an algebra $A$ is the pair of 
a graded $A$-module $C = \bigoplus_{n \in \mathbb{Z}} C_{n}$ and 
a graded $A$-linear map $d^{C}: C \rightarrow C$ of degree $-1$ such that $(d^{C})^{2}=0$.
Dually, a \textit{cochain complex} $C = (C, d_{C})$ is the pair of 
a graded $A$-module $C = \bigoplus_{n \in \mathbb{Z}} C^{n}$ and 
a graded $A$-linear map $d_{C}: C \rightarrow C$ of degree $1$ 
such that $d_{C}^2 = 0$.
In both cases, the graded map $d$ is called the \textit{differential} of $C$.
Note that any cochain complex $C$ can be regarded as a chain complex  
by reindexing $C_{n} = C^{-n}$, and vice versa. 
For a chain complex $C$ and $i \in \mathbb{Z}$, we define 
the $i$-shifted chain complex $\boldsymbol{\Sigma}^{i} C = (\boldsymbol{\Sigma}^{i} C, d^{\boldsymbol{\Sigma}^{i} C})$  

to be the chain complex given by
$(\boldsymbol{\Sigma}^{i} C)_{n} := C_{n-i}$ and $d^{\boldsymbol{\Sigma}^{i} C} := (-1)^{i} d^{C}$.
For two chain complexes $C$ and $C^{\prime}$ of $A$-modules, 
a \textit{chain map} $f: C \rightarrow C^{\prime}$ is a graded $A$-linear map
$f: C \rightarrow C^{\prime}$ of degree $0$ such that $d^{C^{\prime}} f = f d^{C}$. 

For two chain maps $f, g: C \rightarrow C^{\prime}$,
we say that $f$ is \textit{homotopic} to $g$ if there exists 
a graded $A$-linear map $h: C \rightarrow C^{\prime}$ of degree $1$ such that
$f -g = d^{C^{\prime}}h + h d^{C}$.
Then the graded map $h$ is called a \textit{homotopy} from $f$ to $g$.
In such a case, we denote $f \sim g$.
A chain map $f: C \rightarrow C^{\prime}$ is called \textit{null-homotopic} 
if $f \sim 0$.
It is well-known that the relation $\sim$ is an equivalence relation.  
We denote by $[C, C^{\prime}]$ the quotient $k$-module induced by this equivalence relation.
A chain map $f: C \rightarrow C^{\prime}$ is a \textit{homotopy equivalence} if there exists a chain map 
$g: C^{\prime} \rightarrow C$ such that $fg \sim \id_{C^{\prime}}$ and $g f \sim \id_{C}$, and
a chain map $f: C \rightarrow C^{\prime}$ is a \textit{quasi-isomorphism} if the morphism
$\Hh_{n}(f): \Hh_{n}(C) \rightarrow \Hh_{n}(C^{\prime})$ induced by $f$ is an isomorphism for any $n \in \Z$, 
where  $\Hh_{n}(C)$ is the $n$-th homology group of the complex $C$ defined by the quotient $A$-module 
$Z_{n}(C) / B_{n}(C)$ with $Z_{n}(C) =\Ker d^{C}_{n}$ and $B_{n}(C) = \Im d^{C}_{n+1}$.
We say that a chain complex $C$ is \textit{acyclic} if $\Hh_{n}(C) = 0$ for every $n \in \Z$ and that
a chain complex $C$ is \textit{contractible} if $id_{C} \sim 0$.
Then a homotopy from $\id_{C}$ to $0$ is called a \textit{contracting homotopy}. 

The following easy and well-known lemma and its dual lemma play important roles in this paper. 
We will include a proof of the first lemma.
%
%
\begin{lemma}[{\cite[Chapter I, Lemma 7.4]{Brown94}}] \label{lem:7}
Let $C$ and $C^{\prime}$ be chain complexes of $A$-modules and $r \in \Z$.
Suppose that $C_{i}$ is projective over $A$ for $i > r$ and that $\Hh_{i}(C^{\prime}) = 0$ for $i \geq r$.
Then 
\begin{enumerate}
    \item[(1)] 
        Any family $\{f_{i}: C_{i} \rightarrow C^{\prime}_{i}\}_{i \leq r}$ commuting with differentials extends to a chain map 
        $f: C \rightarrow C^{\prime}$.
    \item[(2)] 
        Let $f, g: C \rightarrow C^{\prime}$ be chain maps and 
        $\{ h_{i}: C_{i} \rightarrow C^{\prime}_{i+1}\}_{i \leq r}$ a family of $A$-linear maps
        such that
        $d^{C^{\prime}}_{i+1} h_{i} +h_{i-1}d^{C}_{i} = f_{i}-g_{i}$ for $i \geq r$. 
        Then $\{ h_{i}\}_{i \leq r}$ extends to a homotopy from $f$ to $g$.
\end{enumerate} 
\end{lemma}
\begin{proof}
We only show $(1)$; the proof of $(2)$ is similar.
For each $r \in \Z$, let $d_{r}^{C^{\prime}} = \iota_{r}^{\prime} \pi_{r}^{\prime}$ be 
the canonical factorization through $\Im d_{r}^{C^{\prime}}$. 
Since we have $d_{r}^{C^{\prime}} f_{r} d_{r+1}^{C} = 0$ and $\Hh_{r}(C^{\prime}) = 0$,
there exists a lifting morphism $f_{r+1}^{\prime}$ of $f_{r} d_{r+1}^{C}$ such that $f_{r} d_{r+1}^{C} = \iota^{\prime}_{r} f_{r+1}^{\prime}$.
The projectivity of $C_{r+1}$ implies that there exists a morphism $f_{r+1}: C_{r+1} \rightarrow C_{r+1}^{\prime}$ such that
$f^{\prime}_{r+1} = \pi_{r+1}^{\prime} f_{r+1}$, so that we have
\begin{align*}
d_{r+1}^{C^{\prime}} f_{r+1}
=
\iota^{\prime}_{r}  (\pi_{r+1}^{\prime} f_{r+1})
=
\iota^{\prime}_{r} f^{\prime}_{r+1}
=
f_{r} d_{r+1}^{C}.
\end{align*}
Repeating this argument inductively, we obtain the desired chain map.
\end{proof}
%
%

%
%
\begin{lemma} \label{lem:8}
Let $C$ and $C^{\prime}$ be chain complexes of $A$-modules and $r \in \Z$.
Suppose that $C^{\prime}_{i}$ is injective over $A$ for $i < r$ and that $\Hh_{i}(C) = 0$ for $i \leq r$.
Then 
\begin{enumerate}
    \item[(1)] 
        Any family $\{ f_{i}: C_{i} \rightarrow C^{\prime}_{i} \}_{i \geq r}$ of morphisms of $A$-modules commuting with differentials extends to a chain map 
        $f: C \rightarrow C^{\prime}$.
    \item[(2)] 
        Let $f, g: C \rightarrow C^{\prime}$ be chain maps and $\{ h_{i}: C_{i} \rightarrow C^{\prime}_{i+1}\}_{i \geq r}$  
        a family of $A$-linear maps such that
        $d^{C^{\prime}}_{i+1} h_{i} +h_{i-1}d^{C}_{i} = f_{i}-g_{i}$ for $i \geq r$. 
        Then $\{ h_{i} \}_{i \geq r}$ extends to a homotopy from $f$ to $g$.
\end{enumerate} 
\end{lemma}
%
%

Let $C$ and $C^\prime$ be two chain complexes of $A$-modules. 
We define the tensor product $C \otimes_{A} C^{\prime}$ as the chain complex with components 
\begin{align*}
(C \otimes_{A} C^{\prime})_{n} := \bigoplus_{i \in \Z} C_{n-i} \otimes_{A} C^{\prime}_{i}
\end{align*}
and differentials 
$d^{C \otimes_{A} C^{\prime}}_{n} : (C \otimes_{A} C^{\prime})_{n} \rightarrow (C \otimes_{A} C^{\prime})_{n-1}$
given by
\begin{align*}
d^{C \otimes_{A} C^{\prime}}_{n}(x \otimes_{A} x^{\prime}) 
:= d^{C}(x) \otimes x^{\prime} +(-1)^{|x|} x \otimes d^{C^{\prime}}(x^{\prime}) 
\end{align*}
for homogeneous elements $x \in C$ and $x^{\prime} \in C^{\prime}$.

Now we define $\sHom_{A}(C, C^{\prime})$ to be the chain complex with components 
\begin{align*}
\sHom_{A}(C, C^{\prime})_{n} := \prod_{i \in Z} \Hom_{A}(C_{i}, C^{\prime}_{i+n})
\end{align*}
and differentials 
\[
 d^{\sHom_{A}(C, C^{\prime})}_{n} : \sHom_{A}(C, C^{\prime})_{n} \rightarrow \sHom_{A}(C, C^{\prime})_{n-1}
\]
given by
\begin{align*}
 d^{\sHom_{A}(C, C^{\prime})}_{n}(f) 
 := \left(  d^{C^{\prime}}_{i+n} f_{i} -(-1)^{n} f_{i-1} d^{C}_{i} \right)_{i \in \Z}
\end{align*}
for any $f = (f_{i})_{i \in \Z} \in \sHom_{A}(C, C^{\prime})_{n}$. 
One sees that $\Hh _{n}(\sHom_{A}(C, C^{\prime})) = [C, \boldsymbol{\Sigma}^{-n} C^{\prime}]$ 
for any $n \in \Z$. 
We put $[C, C^{\prime}]_{n} := [C, \boldsymbol{\Sigma}^{-n} C^{\prime}] $.

Let $C$ be a chain complex of left $A$-modules. 
We define the \textit{$k$-dual complex} $\D(C)$ of $C$ by the cochain complex $\sHom_{k}(C, k)$ of right $A$-modules.
Moreover, the \textit{$A$-dual complex} $C^{\vee}$ of $C$ is defined as the cochain complex $\sHom_{A}(C, A)$ of right $A$-modules.

Let $C, C^{\prime}$ and $C^{\prime \prime}$ be chain complexes of $A$-modules.
The composition of graded maps is defined to be
the chain map
\[
 \sHom_{A}(C^{\prime}, C^{\prime \prime}) \otimes \sHom_{A}(C, C^{\prime})  \xrightarrow{\circ} \sHom_{A}(C, C^{\prime \prime})
\]
sending 
$u \otimes v \in \sHom_{A}(C^{\prime}, C^{\prime \prime})_{s} \otimes \sHom_{A}(C, C^{\prime})_{r}$
to
\[
u \circ v := (u_{p+r} v_{p})_{p \in \Z} \in \sHom_{A}(C, C^{\prime \prime})_{r+s}.
\]
Hence the chain map induces a well-defined operator
\[
 \Hh_{*}(\sHom_{A}(C^{\prime}, C^{\prime \prime})) \otimes \Hh_{*}(\sHom_{A}(C, C^{\prime}))  \rightarrow \Hh_{*}(\sHom_{A}(C, C^{\prime \prime})).
\]

Let $C_{1}, C_{2}, C^{\prime}_{1}$ and $C^{\prime}_{2}$ be chain complexes of $A$-modules, 
and let $u: C_{1} \rightarrow C^{\prime}_{1}$ and $v: C_{2} \rightarrow C^{\prime}_{2}$ be 
graded maps of degree $s$ and of degree $t$, respectively.
The \textit{tensor product} $u \otimes_{A} v$ of $u$ and $v$ 
is defined as the graded map of degree $s+t$ given by 
\[
(u \otimes_{A} v)(x_{1} \otimes x_{2}) := (-1)^{|v||x_{1}|} u(x_{1}) \otimes_{A} v(x_{2})
\]
for homogeneous elements $x_{1} \in C_{1}$ and $x_{2} \in C_{2}$.
The tensor product of graded maps gives rise to a chain map
\[
\sHom_{A}(C_{1}, C^{\prime}_{1}) \otimes \sHom_{A}(C_{2}, C^{\prime}_{2}) \rightarrow \sHom_{A}(C_{1} \otimes_{A} C_{2}, C^{\prime}_{1} \otimes_{A} C^{\prime}_{2}).
\]

%
%
\subsection{Hochschild (co)homology groups}
Throughout this section and the next, let $A$ be a $k$-algebra which is projective over $k$. 
Let $\Ae$ denote the enveloping algebra $A \otimes A^{\circ}$ of $A$, where $A^{\circ}$ is the opposite algebra of $A$.
Suppose that all projective resolutions of $A$ are taken over $\Ae$.
We view any module $M$ as a complex concentrated in degree $0$. 
Let $\P = \bigoplus_{i \geq 0}P_{i}$ be a projective resolution of $A$, that is, 
a quasi-isomorphism $\P \xrightarrow{\varepsilon} A$ with $P_{i}$ finitely generated projective.
The epimorphism $\varepsilon : P_{0} \rightarrow A$ is called the \textit{augmentation} of $\P$.

The $n$-th \textit{Hochschild cohomology group $\Hh^{n}(A, M)$ of $A$ with coefficients in an $A$-bimodule $M$} is defined by
$\Hh^{n}(A, M) := \Hh^{n}(\sHom_{\Ae}(\P, M))$,
where  the $n$-th component of $\sHom_{\Ae}(\P, M)$ is given by 
\[
\sHom_{\Ae}(\P, M)^{n} =\sHom_{\Ae}(\P, M)_{-n} = \Hom_{\Ae}(P_{n}, M).
\]
The $n$-th \textit{Hochschild homology group $\Hh_{n}(A, M)$ of $A$ with coefficients in an $A$-bimodule $M$} is defined as
$\Hh_{n}(A, M) := \Hh_{n}(\P \otimes_{\Ae} M)$. 

%
%
\subsection{Cup product and cap product in Hochschild theory}
In this section, we recall the definitions of the cup product and of the cap product in Hochschild theory.
If $\P \xrightarrow{\varepsilon} A$ is a projective resolution, we can associate $\P$ with the following augmented chain complex having $A$ in degree $-1$:
\[
\cdots \rightarrow P_{2} \rightarrow P_{1} \rightarrow P_{0} \xrightarrow{\varepsilon} A \rightarrow 0 \rightarrow 0 \rightarrow \cdots.
\]
For two projective resolutions $\P \xrightarrow{\varepsilon} A$ and $\P^{\prime} \xrightarrow{\varepsilon^{\prime}} A$, 
a chain map $f: \P \rightarrow \P^{\prime}$ is called an \textit{augmentation-preserving chain map} if $\varepsilon^{\prime} f_{0} = \varepsilon$.
Since the tensor product $\P \otimes_{A} \P$ of $\P$ with itself is also a projective resolution of $A$ with augmentation 
$\varepsilon \otimes_{A} \varepsilon : P_{0} \otimes_{A} P_{0} \rightarrow A$, there exists an augmentation-preserving chain map $\Delta :\P \rightarrow \P \otimes_{A} \P$.
We call such a  chain map  a \textit{diagonal approximation}.
For any $A$-bimodules $M$ and $N$, we define a graded $k$-linear map
\[
\smile: \sHom_{\Ae}(\P, M) \otimes \sHom_{\Ae}(\P, N) \rightarrow \sHom_{\Ae}(\P, M \otimes_{A} N)
\]
by
\[
u \smile v := (u \otimes_{A} v) \Delta
\]
for homogeneous elements $u \in \sHom_{\Ae}(\P, M)$ and $v \in \sHom_{\Ae}(\P, N)$.
Since the diagonal approximation $\Delta$ is a chain map, 
we see that the map $\smile$ is a chain map. Thus it induces 
a well-defined operator
\[
\smile: \Hh^{*}(A, M) \otimes \Hh^{*}(A, N) \rightarrow \Hh^{*}(A, M \otimes_{A} N).
\]
For $u \in \Hh^{m}(A, M)$ and $v \in \Hh^{n}(A, N)$, we call $u \smile v \in \Hh^{m+n}(A, M \otimes_{A} N)$ the \textit{cup product} of $u$ and $v$. 
%
%
\begin{example} \label{example:1}
{\rm
Let $A$ be an algebra over a field $k$, 
and let $\overline{A}$ denote a quotient vector space $A/(k \cdot 1)$. 
For simplicity, we write $\overline{a}_{1, n}$ for 
$\overline{a}_{1} \otimes \cdots \otimes \overline{a}_{n} \in \overline{A}^{\otimes n}$.
The \textit{normalized bar resolution} $\mathrm{Bar}(A)$ of $A$ 
is the chain complex with 
$\mathrm{Bar}(A)_{n} := A \otimes \overline{A}^{\otimes n} \otimes A$ and differentials
\[
d_{n}^{\mathrm{Bar}(A)}(a_{0} \otimes \overline{a}_{1, n} \otimes a_{n+1})
=
\sum_{i=0}^{n} (-1)^{i} 
a_{0} \otimes \overline{a}_{1, i-1} \otimes \overline{a_{i} a_{i+1}} \otimes \overline{a}_{i+1, n} \otimes a_{n+1}.
\]
Then $\mathrm{Bar}(A)$ is a projective resolution of $A$.
It is known that the graded $\Ae$-linear map $\Delta: \mathrm{Bar}(A) \rightarrow \mathrm{Bar}(A) \otimes_{A} \mathrm{Bar}(A)$ 
given by
\begin{align*}
    \Delta(a_{0} \otimes&\,\overline{a}_{1, n} \otimes a_{n+1}) 
    :=
    \sum_{i=0}^{n} (a_{0} \otimes \overline{a}_{1, i} \otimes 1) \otimes_{A} (1 \otimes \overline{a}_{i+1, n} \otimes a_{n+1})
\end{align*}
is a chain map. Then we see that
\begin{align*}
    (u \smile v)&(a_{0} \otimes \overline{a}_{1, m+n} \otimes a_{m+n+1})\\
    &=
    (-1)^{m n} u(a_{0} \otimes \overline{a}_{1, m} \otimes 1) \otimes_{A} v(1 \otimes \overline{a}_{m, m+n} \otimes a_{m+n+1}) 
\end{align*}
for $u \in \sHom_{\Ae}(\mathrm{Bar}(A), M)^{m}$ and $v \in \sHom_{\Ae}(\mathrm{Bar}(A), N)^{n}$.
}
\end{example}
%
%

Consider the chain map
\[
\sHom_{\Ae}(\P, M) \otimes \sHom_{\Ae}(\P, \P \otimes_{A} N) \rightarrow  \sHom_{\Ae}(\P, M \otimes_{A} N)
\]
defined by $u \otimes v \mapsto (u \otimes_{A}\id_{M}) v$ for homogeneous elements $u \in \sHom_{\Ae}(\P, M)$ and 
$v \in \sHom_{\Ae}(\P, \P \otimes_{A} N)$.
Then it induces a well-defined operator, called the \textit{composition product},
\begin{align} \label{eq:15}
 \Hh^{*}(A, M) \otimes \Hh^{*}(\sHom_{\Ae}(\P, \P \otimes_{A} N)) \rightarrow  \Hh^{*}(A, M \otimes_{A} N).
\end{align}
Since $A$ is projective as a right $A$-module, the augmented chain complex $\P \xrightarrow{\varepsilon} A$ is contractible 
as a complex of right $A$-modules, so that the tensor product $\P \otimes_{A} N \rightarrow A \otimes_{A} N$ is acyclic, 
which means that $\varepsilon \otimes_{A} \id_{N}: \P \otimes_{A} N \rightarrow N$ is a quasi-isomorphism.
It follows from \cite[Chapter I, Theorem 8.5]{Brown94} that $\varepsilon \otimes_{A} \id_{N}$ induces an isomorphism 
\begin{align*} 
    \Hh^{n}(A, N) \cong \Hh^{n}(\sHom_{\Ae}(\P, \P \otimes_{A} N)).
\end{align*}
%
%
\begin{theo}[{\cite[Proposition 1.1]{BuchGreEdwaSnasSolb08}}] \label{theo:5}
    The cup product
\[
\smile: \Hh^{*}(A, M) \otimes \Hh^{*}(A, N) \rightarrow \Hh^{*}(A, M \otimes_{A} N).
\]
 coincides with the composition product $(\ref{eq:15})$ via the isomorphism above.
\end{theo}

We now recall the definition of the \textit{cap product} between Hochschild cohomology and homology groups and 
the statement analogue to Theorem $\ref{theo:5}$.
Consider the chain map 
\[
\gamma : \sHom_{\Ae}(\P, M) \otimes ( (\P \otimes_{A} \P )\otimes_{\Ae} N ) \rightarrow \P \otimes_{\Ae} (M \otimes_{A} N)
\]
given by
\begin{align*}
    \gamma(u \otimes x \otimes_{A} y \otimes_{\Ae} n ) := (-1)^{|u||x|} x \otimes_{\Ae} u(y) \otimes_{A} n,
\end{align*} 
where $u \in \sHom_{\Ae}(\P, M)$ and $n \otimes_{\Ae} x \otimes_{A} y \in ( (\P \otimes_{A} \P )\otimes_{\Ae} N )$ 
are homogeneous elements and we use an isomorphism 
\begin{align*}
 (\P \otimes_{A} M )\otimes_{\Ae} N \cong \P \otimes_{\Ae}( M \otimes_{A} N).
\end{align*} 
Then the chain map
\[
\frown: \sHom_{\Ae}(\P, M) \otimes (\P \otimes_{\Ae} N) \rightarrow  \P \otimes_{\Ae} (M \otimes_{A} N)
\]
defined to be the composition of two chain maps $\gamma$ and
\begin{align*}
\id &\otimes (\Delta \otimes_{\Ae} \id): \\
&\sHom_{\Ae}(\P, M) \otimes (\P \otimes_{\Ae} N) \rightarrow \sHom_{\Ae}(\P, M) \otimes ( (\P \otimes_{A} \P) \otimes_{\Ae} N)
\end{align*}
gives rise to a well-defined operator
\[
\frown: \Hh^{m}(A, M) \otimes \Hh_{n}(A, N) \rightarrow \Hh_{n-m}(A, M \otimes_{A} N).
\]
For $u \in \Hh^{m}(A, M)$ and $w \in \Hh_{n}(A, N)$, we call $u \frown w \in \Hh_{n-m}(A, M \otimes_{A} N)$ the \textit{cap product} of $u$ and $w$. 

%
%
\begin{example} \label{example:2}
{\rm
Using the same projective resolution and diagonal approximation as in Example $\ref{example:1}$, we see that
\begin{align*}
    &u \frown w =
    (-1)^{m(p-m)} 
    (a_{0} \otimes \overline{a}_{1, p-m} \otimes 1) \otimes_{\Ae} 
    u(1 \otimes \overline{a}_{p-m+1, p+1}) \otimes_{A} n,
\end{align*}
for 
$w = (a_{0} \otimes \overline{a}_{1, p} \otimes a_{p+1}) \otimes_{\Ae} n 
\in (\mathrm{Bar}(A) \otimes_{\Ae} N)_{p}$ and
$u \in \sHom_{\Ae}(\mathrm{Bar}(A), M)^{m}$ 
with $p-m \geq 0$.
}
\end{example}
%
%

On the other hand, we define a chain map 
\begin{align}
\sHom_{\Ae}(\P, M) \otimes (\P \otimes_{\Ae} (\P \otimes_{A} N)) \rightarrow \P \otimes_{\Ae} (M \otimes_{A} N)
\label{eq:1}
\end{align}
by 
\[
u \otimes (x \otimes_{\Ae} (y \otimes_{A} n)) \mapsto (-1)^{|u||x|} x \otimes_{\Ae} (u(y) \otimes_{A} n)
\]
for homogeneous elements $u \in \sHom_{\Ae}(\P, M)$ and $(x \otimes_{A} y) \otimes_{\Ae} n  \in (\P \otimes_{A} \P) \otimes_{\Ae} N$. 
Moreover, it follows from \cite[Chapter I, Theorem 8.6]{Brown94} that there exists an isomorphism 
\begin{align} \label{eq:14}
\Hh_{*}(A, N) \cong \Hh_{*}(\P \otimes_{\Ae} (\P \otimes_{A} N)).
\end{align}
One proves the following in a similar way to the proof of Theorem $\ref{theo:3} (2)$ below. 

%
%
\begin{theo}
    The cap product
\[
\frown: \Hh^{*}(A, M) \otimes \Hh_{*}(A, N) \rightarrow \Hh_{*}(A, M \otimes_{A} N).
\] 
agrees with the product induced by the chain map $(\ref{eq:1})$ via the isomorphism $(\ref{eq:14})$.
\end{theo}

\noindent Remark that the case for $M = N = A$ is proved by Armenta \cite[Section 4]{Armenta19}.

%
%
\section{Tate-Hochschild (co)homology groups of a self-injective algebra}
\label{sec:2}
Our aim in this section is to recall the definitions of complete resolutions over finite dimensional self-injective algebras and of Tate and Tate-Hochschild (co)homology groups.
Moreover, we provide a characterization of minimal complete resolutions.
Throughout this section, assume that $A$ is
a finite dimensional algebra over a field $k$. 
%
%
\subsection{Twisted bimodules}
Let us begin with the preparation for some notation.
We denote by $\Aut(A)$ the group of algebra automorphisms of $A$.
Note that any $\alpha \in \Aut(A)$ gives rise to $\alpha^\circ \in \Aut(A^{\circ})$ defined by
$\alpha^\circ (a^\circ) := \alpha(a)^\circ$ for any $a^\circ \in A^\circ$.
For an $A$-bimodule $M$ and two automorphisms $\alpha, \beta \in \Aut(A)$, 
we denote by $_\alpha M_\beta$ 
the $A$-bimodule which is $M$ as a $k$-module and whose $A$-bimodule structure is
given by $a*m*b := \alpha(a) m \beta(b)$ for $a, b \in A$ and $m \in {}_\alpha M_\beta$. 
We denote $_{1}M_\beta := {}_{\id}M_\beta$ and ${}_{\alpha}M_{1} := {}_{\alpha}M_{\id}$.
Recall that we can identify an $A$-bimodule $M$ with the left (right) 
$\Ae$-module $M$ of which the structure is defined by
$(a \otimes b^\circ) m:= amb \ (m (a \otimes b^\circ) := bma)$ for $a \otimes b^\circ \in \Ae$ 
and $m \in M$.
Using this identification, we have ${}_{\alpha \otimes \beta}M = {}_\alpha M_\beta = M_{\beta \otimes \alpha}$,
where we set 
$\alpha \otimes \beta := \alpha \otimes \beta^{\circ}$ and 
$\beta \otimes \alpha =: \beta \otimes \alpha^{\circ}$.
For any morphism $f: M \rightarrow N$ of $A$-bimodules and $\alpha$, $\beta \in \Aut(A)$,
there exists an isomorphism
\begin{align*}
\Hom_{\Ae}(M, N) \rightarrow \Hom_{\Ae}({}_\alpha M_\beta, {}_\alpha N_\beta);\ f \mapsto f,
\end{align*}
which is natural in both $M$ and $N$.

%
%
\subsection{Self-injective algebras and Frobenius algebras}
In this subsection, we recall the definitions of self-injective algebras and of Frobenius algebras.
Recall that a finite dimensional algebra $A$ is a \textit{self-injective algebra} if $A$ is injective as a left or as a right $A$-module, or
equivalently, $A$ is injective as a left and as a right $A$-module.
Moreover, recall that a finite dimensional $k$-algebra $A$ is a \textit{Frobenius algebra} if 
there exists a non-degenerate bilinear form $\langle -, - \rangle : A \otimes A \rightarrow k$ satisfying 
$\langle a b, c \rangle = \langle a, b c \rangle$ for $a, b$ and $c \in A$. 
The bilinear form gives rise to a left and a right $A$-module isomorphism
\begin{align}
\label{eq:2}
t_{1} : {}_{A}A \rightarrow {}_{A}D(A); \ x \mapsto \langle -, x \rangle, \\[3pt]
 A_{A} \rightarrow D(A)_{A}; \ x \mapsto \langle x, - \rangle, \nonumber
\end{align}
where the left and the right $A$-module structure of $D(A)$ is given by 
\begin{align*}
    (a f)(x) := f(x a) \mbox{ and } (f a)(x) := f(a x) 
\end{align*}
for $a, x \in A$ and $f \in D(A)$. 
In particular, a Frobenius algebra is a self-injective algebra.
Let $\{ u_{1}, \ldots, u_{r} \}$ be a $k$-basis of $A$. Then we have another $k$-basis $\{ v_{1}, \ldots, v_{r} \}$ such that 
$\langle v_{i}, u_{j} \rangle = \delta_{ij}$ for all $1 \leq i, j \leq r$, where $\delta_{i j}$ denotes the Kronecker delta. 
We call $\{ v_{i}\}_{i}$ the \textit{dual basis} of $\{ u_{i}\}_{i}$.
It is known that there exists an algebra automorphism $\nu$ of $A$ such that 
$\langle a, b \rangle = \langle b, \nu(a) \rangle$ for $a, b \in A$. 
The automorphism $\nu$ is unique, up to inner automorphism, and we call it the \textit{Nakayama automorphism} of $A$.
The Nakayama automorphism $\nu$ of $A$ makes the left $A$-module isomorphism (\ref{eq:2}) into an $A$-bimodule isomorphism 
\[
{}_{1}A_{\nu} \rightarrow D(A).
\]
Moreover, there exists another $A$-bimodule isomorphism 
\[
t_{2} : {}_{1}A_{\nu^{-1}} \rightarrow \Hom_{\Ae}({}_{\Ae}A, {}_{\Ae}\Ae) = A^{\vee}; \quad x \mapsto [y \mapsto \sum_{i} y u_{i} \nu(x) \otimes v_{i}],
\]
where the $A$-bimodule structure of $A^{\vee}$ is given by 
$(a f b )(x) := f(x) (b \otimes a^\circ)$
for $f \in A^{\vee}$ and $ a,b \in A$.

%
%
\subsection{Complete resolutions and their minimalities}
Our aim in this subsection is to recall the definition of complete resolutions and to characterize minimal complete resolutions in terms of minimal projective resolutions and minimal injective resolutions.
Let us start with the definition of complete resolutions.
%
%
\begin{defi}[\cite{AvraMart02}] ~\\[-15pt] 
\begin{enumerate}
\item[(1)]
        A complete resolution of a finitely generated $A$-module $M$ is a diagram 
        \[\T = \bigoplus_{i \in \Z} T_{i} \xrightarrow{\vartheta} \P \xrightarrow{\varepsilon} M,
        \]
        where $\T$ is an exact sequence of finitely generated projective $A$-modules with $\Hh^{i}(\T^{\vee}) =0$ for all $i \in \Z$, $\varepsilon :\P \rightarrow M$ is a projective resolution and $\vartheta : \T \rightarrow \P$ is a chain map such that 
        $\vartheta_{i}$ is an isomorphism for $i \gg 0$.
\item[(2)]
        A finitely generated $A$-module $G$ is \textit{totally reflexive} if the canonical morphism $G \rightarrow (G^{\vee})^{\vee}$ 
        is an isomorphism and $\Ext_{A}^{n} (G, A) = 0 = \Ext_{A^{\textrm{o}}}^{n}(G^{\vee}, A)$ for all $n \geq 1$. 
\item[(3)]
        A $\mathcal{G}$-resolution $($of length $\leq l)$ of a finitely generated $A$-module $M$ is a quasi-isomorphism $\mathbf{G} = \bigoplus_{i \geq 0} G_{i} \rightarrow M$ 
        with $G_{i}$ totally reflexive $($and $G_{i} = 0$ for $i >l)$.
\end{enumerate}
\end{defi}

We define the \textit{G-dimension} $\mathrm{G}$-$\dim_{A} M$ of a finitely generated $A$-module $M$ by
\[
\mathrm{G}\text{-}\dim_{A} M := \inf \left\{ g \in \N \big| \text{there exists a }\mathcal{G}\text{-resolution of length} \leq g \right\}.
\]
Since we are interested in self-injective algebras including Frobenius algebras, we mainly deal with self-injective algebras.
For more general cases, we refer to \cite{AvraMart02,BerghJorgen13}. 

Let $A$ be a self-injective algebra.
Since $A$ is injective as a left and as a right $A$-module, 
any finitely generated $A$-module $M$ is totally reflexive and hence of G-dimension $0$.
It follows from \cite[Theorem 3.1]{AvraMart02} that $M$ has a complete resolution $\T \xrightarrow{\vartheta} \P \rightarrow M$ 
with $\vartheta_{i}$ isomorphic for $i \geq 0$.
Thus, any complete resolution $\T$ of $M$ consists of some projective resolution of $M$ in non-negative degrees, and 
we have $M = \Coker d^{\T}_{1}$. 
Thus we simply write $\T \rightarrow M$ for $\T \rightarrow \P \rightarrow M$.

For a complete resolution $\T$ of $M$ and $i \in \Z$, 
let $d_{i}^{\T} = \iota_{i} \pi_{i}$ be the canonical factorization through $\Im d_{i}^{\T}$,
i.e., $\pi_{i}$ is the canonical epimorphism $T_{i} \rightarrow \Im d_{i}^{\T}$ and 
$\iota_{i}$ is the canonical inclusion $\Im d_{i}^{\T} \hookrightarrow T_{i-1}$.
In particular, we denote by $\varepsilon$ the epimorphism $\pi_{0}: T_{0} \rightarrow \Im d_{0}^{\T} = M$ and by $\eta$ the canonical inclusion 
$\iota_{0}: M \hookrightarrow T_{-1}$.
The morphism $\varepsilon: T_{0} \rightarrow M$ is called the \textit{augmentation} of $\T$.
Note that the augmentation of any complete resolution of $M$ induce a chain map $\T \rightarrow M$, but it is not a quasi-isomorphism.
A chain map $f: \T \rightarrow \T^{\prime}$ between two complete resolutions $\T \xrightarrow{\varepsilon} M$ and 
$\T^{\prime} \xrightarrow{\varepsilon^{\prime}} M$  is called \textit{augmentation-preserving} if $\varepsilon^{\prime} f = \varepsilon$.
It follows from \cite[Lemma 5.3]{AvraMart02} that any augmentation-preserving chain map $f: \T \rightarrow \T^{\prime}$ between two complete resolutions $\T$ and $\T^{\prime}$ of $M$ is a homotopy equivalence. 

In \cite{AvraMart02}, Avramov and Martsinkovsky introduced the notion of minimal complexes.
Recall that a chain complex $C$ over $A$ is \textit{minimal} if
every homotopy equivalence $C \rightarrow C$ is an isomorphism. 
Clearly, the minimality of a complex preserves under taking shifts. 
We will apply the notion to complete resolutions over a self-injective algebra.
%
%
\begin{defi}
Let $A$ be a self-injective algebra and $M$ a finitely generated $A$-module. 
Then a complete resolution $\T \rightarrow \P \rightarrow M$ is called \textit{minimal} if $\T$ is minimal.
\end{defi}
%
%
\noindent
Remark that our definition of minimal complete resolutions does not require the minimalities 
of the projective resolutions in non-negative degrees.
We now characterize minimal complete resolutions of $M$ in terms of its projective and injective resolutions. 
For this purpose, we first recall the result of Avramov and Martsinkovsky.
Recall that a projective resolution $\P \rightarrow M$
is a \textit{minimal projective resolution} if $P_{n}$ is a projective cover of $\Coker d_{n+1}^{\P}$ for all $n \geq 0$ and 
that an injective resolution $M \rightarrow \I = \bigoplus_{i \leq 0}I_{i}$ is 
a \textit{minimal injective resolution} if $I_{n}$ is an injective envelope of $\Ker d_{n}^{\I}$ for all $n \leq 0$.
%
%
\begin{lemma}[{\cite[Example 1.8]{AvraMart02}}] \label{lem:9}
Let $M$ be a finitely generated $A$-module, and let 
$\P \rightarrow M$ be a projective resolution and $M \rightarrow \I $ an injective resolution.
Then the following statements hold.
\begin{enumerate}
\item[(1)] 
    $\P$ is minimal if and only if $P_{n}$ is a projective cover of $\Coker d_{n+1}^{\P}$ for all $n \geq 0$.
\item[(2)] 
    $\I$ is minimal if and only if $I_{n}$ is an injective envelope of $\Ker d_{n}^{\I}$ for all $n \leq 0$.
\end{enumerate}
\end{lemma}
\begin{proof}
we only prove $(1)$; the proof of $(2)$ is similar.
It follows from \cite[Proposition 1.7(1)]{AvraMart02} that $\P$ is minimal if and only if 
each chain map $f: \P \rightarrow \P$ homotopic to $\id_{\P}$ is an isomorphism. 
Take a chain map $f: \P \rightarrow \P$ such that $f \sim \id_{\P}$.
Then there exists a morphism $h_{0}: P_{0} \rightarrow P_{1}$ such that $\id_{P_{0}} - f_{0} = d_{1}h_{0}$.
Letting $\varepsilon$ be the augmentation $P_{0} \rightarrow M$, we have
$\varepsilon (\id -f_{0})= \varepsilon (d_{1}h_{0}) = 0$.
Since $\varepsilon$ is a projective cover of $M$, the morphism $f_{0}$ is an epimorphism and hence an isomorphism.
Moreover, it induces a commutative square
\[\xymatrix@M=6pt{
P_{1} \ar@{->>}[r] \ar[d]_-{f_{1}} &
\Coker d^{\P}_{2} \ar[d]^-{\overline{f_{1}}} \\
P_{1} \ar@{->>}[r] &
\Coker d^{\P}_{2}
}\]
where the  morphism $\overline{f_{1}}$ induced by $f_{1}$ is an isomorphism and 
the horizontal morphisms are the canonical epimorphisms.
Since the canonical epimorphism $P_{1} \rightarrow \Coker d^{\P}_{2}$ is a projective cover,   
the morphism $f_{1}$ is an isomorphism.
Inductively, we see that the morphism $f_{i}$ is an isomorphism for all $i \geq 0$. 

Conversely, suppose that $\P$ is minimal. 
Observe that any projective resolution $\P$ of $M$ can be decomposed as $\P = \P_{M} \oplus \P^{\prime}$, 
where $\P_{M}$ is a minimal projective resolution of $M$ and $\P^{\prime}$ is a contractible complex. 
It follows from \cite[Proposition 1.7(3)]{AvraMart02} that $\P^{\prime}$ must be zero. 
This completes the proof. 
\end{proof}
%
%

Let $A$ be a self-injective algebra and $\T$ a complete resolution of a finitely generated $A$-module $M$ 
with $d_{0}^{\T} = \eta \varepsilon$.
Let $\T_{\geq 0} := \bigoplus_{i \geq 0} T_{i}$ and $\T_{< 0} :=\bigoplus_{i < 0} T_{i}$ be 
the truncated subcomplexes of $\T$ with the inherited differentials, which are of the forms
\[\xymatrix@C=15pt@R=10pt{
    \T_{\geq 0}: \cdots \ar[r] & T_{2} \ar[r]^-{d_{2}} & T_{1} \ar[r]^-{d_{1}} & T_{0} \ar[r] & 0 \ar[r] & 0 \ar[r] & 0 \ar[r] & \cdots, \\
    \T_{< 0}: \cdots \ar[r] &  0  \ar[r] & 0 \ar[r] & 0 \ar[r] & T_{-1} \ar[r]^-{d_{-1}} & T_{-2} \ar[r]^-{d_{-2}} & T_{-3} \ar[r] & \cdots.
}\]
Note that $\T_{\geq 0}$ is a projective resolution of $M$ and that $\T_{< 0}$ is isomorphic to $\I[-1]$ 
for some injective resolution $\I$ of $M$.

%
%
\begin{prop} \label{prop:8}
Under the same notation above, the following statements are equivalent.
\begin{enumerate}
\item 
    $\T$ is a minimal complete resolution of $M$.
\item 
    $T_{n}$ is a projective cover of $\Coker d_{n+1}^{\T}$ for all $n \in \Z$.
\end{enumerate}
\end{prop}
\begin{proof}
Assume that $\T$ is minimal.
Since the $(-n)$-shifted complex $\T[-n]$ is a minimal complete resolution of $\Omega_{A}^{n}(M)$
for any $n \in \Z$,
it suffices to show that the augmentation $\varepsilon :P_{0} \rightarrow M$ is a projective cover.
Let $N$ be an $A$-module and $f: N \rightarrow P_{0}$ be a morphism such that
the composite $\varepsilon f$ is an epimorphism.
The projectivity of $P_{0}$ implies that there exists a morphism $g: P_{0} \rightarrow N$ such that 
$\varepsilon = \varepsilon (f g)$. By Lemmas $\ref{lem:7}$ and $\ref{lem:8}$, the morphism $fg$ extends to a chain map
$\varphi: \T \rightarrow \T$ satisfying $\varepsilon \varphi_{0} = \varepsilon fg = \varepsilon$.
It follows from \cite[Lemma 5.3]{AvraMart02} that the chain map $\varphi$ is homotopy equivalent.
The minimality of $\T$ implies that $\varphi_{0}= f g$ is an isomorphism. Therefore, $f$ is an epimorphism.

Conversely, thanks to \cite[Proposition 1.7(1)]{AvraMart02}, 
it suffices to prove the converse for a chain map 
$f: \T \rightarrow \T$ such that $f \sim \id_{\T}$.
Take a homotopy $h$ from $\id_{\T}$ to $f$ and define a graded map $\varphi =( \varphi_{i} )_{i \in \Z}: \T \rightarrow \T$ by
\[
\varphi_{i}=
    \begin{cases}
    f_{i} & \mbox{ if $i \not = 0, -1$,} \\[3pt]
    f_{0} + h_{-1} d_{0}^{\T} & \mbox{ if $i = 0$,} \\[3pt]
    f_{-1} + d_{0}^{\T} h_{-1} & \mbox{ if $i = -1$}. 
    \end{cases}
\]
Then $\varphi$ is a chain map such that $\varepsilon \varphi_{0} = \varepsilon$.
Note that $\T_{\geq 0}$ is a minimal projective resolution of $M$.
Since the chain map $\varphi_{\geq 0}: \T_{\geq 0} \rightarrow \T_{\geq 0}$ is homotopy equivalent,
it follows from Lemma $\ref{lem:9}$(1) that each $\varphi_{i} = f_{i}$ with $i > 0$ is an isomorphism.
Since the inclusion $\iota_{n}$ with $n \leq 0$ is an injective envelope of $\Ker d_{n-1}^{\T}$,
Lemma $\ref{lem:9}$(2) and the fact that $\varphi_{-1} \eta = \eta$ yield that 
each $\varphi_{i}=f_{i}$ with $i < -1$ is an isomorphism. 
It remains to show that $f_{0}$ and $f_{-1}$ are isomorphisms.
Since $f_{1}$ is an isomorphism, so is the restriction $\widetilde{f_{0}}$ of $f_{0}$ to 
$\Ker d_{0}^{\T}$. In a commutative square
\[\xymatrix@C=25pt@M=8pt{
\Ker d_{0}^{\T} \ar@{^{(}->}[r]^-{\iota_{1}} \ar[d]^-{\widetilde{f_{0}}}&
T_{0} \ar[d]^-{f_{0}}
\\
\Ker d_{0}^{\T} \ar@{^{(}->}[r]^-{\iota_{1}} &
T_{0}
}\]
the fact that $\iota_{1}$ is an injective envelope of $\Ker d_{0}^{\T}$ implies that 
$f_{0}$ is a monomorphism and thus an isomorphism.
Similarly, one shows that $f_{-1}$ is an isomorphism.
\end{proof}
%
%

In the course of the proof of Proposition $\ref{prop:8}$, we have proved the following.

%
%
\begin{cor} \label{cor:6}
Let $A$ be a self-injective algebra and $M$ a finitely generated $A$-module. 
Then any minimal complete resolution of $M$ is isomorphic to the complete resolution of the form
\[\xymatrix@R=15pt{
\P \ar[rr] \ar[rd]_-{\varepsilon} \ar@{}[rrd]|(.45){\circlearrowright}& 
&
\boldsymbol{\Sigma}^{-1}\I \\
&
M \ar[ru]_-{-\eta}&
}\]
where $\P \xrightarrow{\varepsilon} M$ is a minimal projective resolution and $M \xrightarrow{\eta} \I$
is a minimal injective resolution.
\end{cor}
%
%
It follows from Corollary $\ref{cor:6}$ that any finitely generated $A$-module admits a minimal complete resolution.
It follows from \cite[Proposition 1.7(2)]{AvraMart02} that 
a minimal complete resolution is uniquely determined up to isomorphism.
For a minimal complete resolution $\T$ of a finitely generated $A$-module $M$ and $i \in \Z$, 
we set $\Omega_{A}^{i} M := \Ker d_{i-1}^{\T}$.
Note that Corollary $\ref{cor:6}$ implies that the module $\Omega_{A}^{i} M$ is 
nothing but the syzygy module of $M$ if $i \geq 0$ and the cosyzygy module of $M$ if $i \leq -1$ 
(see \cite{SkowYama11} for (co)syzygy modules).

%
%
\subsection{Tate and Tate-Hochschild (co)homology groups}
In this subsection, we recall the definition of Tate and Tate-Hochschild (co)homology groups and show that there exists certain duality between Tate-Hochschild cohomology and homology groups. 
Recall that if $A$ is a self-injective (Frobenius) algebra, then so is the enveloping algebra $\Ae$. 
%
%
\begin{defi}
Let $A$ be a  self-injective algebra and $n \in \Z$.
\begin{enumerate}
\item[(1)]
        Let $M$ be a finitely generated left $A$-module with a complete resolution $\T^{M}$, 
        $L$ a finitely generated right $A$-module with a complete resolution $\T^{L}$ 
        and $N$ a left $A$-module.
        The $n$-th Tate cohomology group $\hExt_{A}^{n}(M, N)$ is defined by
        \[
        \hExt_{A}^{n}(M, N) := \Hh^{n}(\sHom_{A}(\T^{M}, N)),
        \]
        and the $n$-th Tate homology group $\hTor^{A}_{n}(M, N)$ is defined by 
        \[
        \hTor_{n}^{A}(L, N) := \Hh_{n}(\T^{L} \otimes_{A} N).
        \]
\item[(2)]
        The $n$-th Tate-Hochschild cohomology group $\hHh^{n}(A, N)$ of $A$ with coefficients in an $A$-bimodule $N$ is defined by
        \[
        \hHh^{n}(A, N) := \hExt_{\Ae}^{n}(A, N).
        \]
        The $n$-th Tate-Hochschild homology group $\hHh_{n}(A, N)$ of $A$ with coefficients in $N$ by 
        \[
        \hHh_{n}(A, N) :=\hTor_{n}^{\Ae}(A, N).
        \]
\end{enumerate}
\end{defi}

Let us recall the definitions of projectively stable categories and of stable tensor products: the \textit{projectively stable category} $A$-$\underline{\mathrm{mod}}$ of $A$ is the category
whose objects are finitely generated left $A$-modules and whose morphisms are given by the quotient space
\begin{align*}
    \underline{\Hom}_{A}(M, N) := \Hom_{A}(M, N)/\mathcal{P}(M, N),
\end{align*}
where $\mathcal{P}(M, N)$ is the space of morphisms factoring through a projective $A$-module.
The \textit{stable tensor product} $L \underline{\otimes}_{A} N$ of a finitely generated right $A$-module $L$ with a finitely generated left $A$-module $N$ is defined to be  
\begin{align*}
    L \underline{\otimes}_{A} N 
    := 
    \Bigg\{ w \in L \otimes_{A} N \ \Bigg|\  
    w \in \bigcap_{\substack{\iota:N \hookrightarrow I \\ I: {\footnotesize \mbox{injective}}}} \Ker( \id_{L} \otimes_{A} \iota) \Bigg\}.
\end{align*}
In case of Frobenius algebras, it follows from \cite[Proposition 2.1.3]{EuSched09} that there exists an isomorphism
\begin{align*}
    L \underline{\otimes}_{A} N 
    \cong 
    \Bigg\{ w \in L \otimes_{A} N \ \Bigg|\  
    w \in \bigcap_{\substack{\iota: L \hookrightarrow I \\ I: {\footnotesize \mbox{injective}}}} \Ker( \iota  \otimes_{A} \id_{N}) \Bigg\}.
\end{align*}
We will identify the two modules above via this isomorphism.
Remark that \cite[Proposition 2.1.3]{EuSched09} also holds for self-injective algebras, 
because the key of the proof is the projectivity of an injective module.

It is well-known that there exist isomorphisms
\begin{align*}
    &\Ext_{A}^{i}(M, N) \cong \underline{\Hom}_{A}(\Omega_{A}^{i}M, N), \quad
    \Tor_{i}^{A}(L, N) \cong \Omega_{A}^{i} L \underline{\otimes}_{A} N
\end{align*}
for $i \geq 1$ when $A$ is a self-injective algebra. 
It is known that there exist such isomorphisms even for Tate (co)homology groups.
We will include a proof.
%
%
\begin{prop} \label{prop:9}
Let $A$ be a self-injective algebra, $L$ a finitely generated right $A$-module and $M$ and $N$ finitely generated left $A$-modules.
Then there exist isomorphisms
\begin{align*}
    &\hExt_{A}^{i}(M, N) \cong \underline{\Hom}_{A}(\Omega_{A}^{i}M, N), \quad
    \hTor_{i}^{A}(L, N) \cong \Omega_{A}^{i} L \underline{\otimes}_{A} N
\end{align*}
for $i \in \Z$.
\end{prop}
\begin{proof}
Let $i \in \Z$ be fixed. For the first isomorphism, let $\T$ be a minimal complete resolution of $M$. 
Take $f \in \Hom_{A}(\Omega_{A}^{i}M, N)$ and consider a commutative diagram
\[\xymatrix@R=10pt{
T_{i} \ar[rr]^-{d_{i}^{\T}} \ar[rd]_-{\pi_{i}} \ar[dd]_-{ f \pi_{i}}& 
&
T_{i-1} \\
&
\Omega_{A}^{i}M \ar[ur]_-{\iota_{i}} \ar[ld]^-{f}&
\\
N&
&
}\]
Obviously, $f\pi_{i}$ belongs to $\Ker d^{\sHom_{A}(\T, N)}_{i}$. 
If $f$ factors through a projective $A$-module $P$, then $f \pi_{i} \in \sHom_{A}(\T, N)^{i}$ is a coboundary, because the projective module $P$ is injective.
Thus we have a well-defined morphism
\begin{align*}
\Phi_{i}: \underline{\Hom}_{A}(\Omega_{A}^{i}M, N) \rightarrow \hExt_{A}^{i}(M, N)
\end{align*}
given by $[f] \mapsto [f \pi_{i}]$.
We claim that $\Phi_{i}$ is an isomorphism.
If $g \in \sHom_{A}(\T, N)^{i}$ lies in $\Ker d^{\sHom_{A}(\T, N)}_{i}$, then there uniquely exists a morphism 
$g^{\prime}: \Omega_{A}^{i}M \rightarrow N$ such that $g = g^{\prime} \pi_{i}$. This implies that $\Phi_{i}$ is an epimorphism.
Assume now that $f \pi_{i}$ is a coboundary, that is, $f \pi_{i} = (-1)^{i} f^\prime d^{\T}_{i}$ for some $f^\prime :T_{i-1} \rightarrow N$.
The surjectivity of $\pi_{i}$ yields that $(-1)^i f^\prime \iota = f$, which means that  $\Phi_{i}$ is a monomorphism.

For the second isomorphism, let $\T$ be a minimal complete resolution of $L$.
Consider a commutative diagram
\[\xymatrix@R=10pt@C=20pt{
T_{i+1} \otimes_{A} N \ar[rr]^-{d_{i+1}^{\T} \otimes_{A} \id} \ar[rd]_(.35){\pi_{i+1} \otimes_{A}\id }& 
&
T_{i} \otimes_{A} N \ar[rr]^-{d_{i}^{\T} \otimes_{A} \id} \ar[rd]_(.35){\pi_{i} \otimes_{A}\id} &
&
T_{i-1} \otimes_{A} N \\
& 
\Omega_{A}^{i+1}L \otimes_{A} N \ar[ru]_(.65){\iota_{i+1} \otimes_{A} \id } &
&
\Omega_{A}^{i}L \otimes_{A} N \ar[ru]_(.65){\iota_{i} \otimes_{A} \id } &
}\]
Let $w \in \Omega_{A}^{i} L \underline{\otimes}_{A} N$ be arbitrary.
Since $\pi_{i} \otimes_{\Ae} \id_{N}$ is an epimorphism, there exists $z \in T_{i} \otimes_{A} N$ such that 
$w =(\pi_{i} \otimes_{\Ae} \id_{N})(z)$.
By the definition of $\Omega_{A}^{i} L \underline{\otimes}_{A} N$ and the injectivity of $T_{i-1}$, we have
$(d_{i}^{\T} \otimes_{A} \id_{N})(z) = (\iota_{i} \otimes_{A} \id_{N})(w) = 0$.
If we take $z^\prime \in T_{i} \otimes_{A} N$ such that $w =(\pi_{i} \otimes_{\Ae} \id_{N})(z^\prime)$, then 
$z-z^\prime$ belongs to $\Im (d_{i+1}^{\T} \otimes_{A} \id_{N})$ because $z-z^\prime \in \Ker (\pi_{i} \otimes_{A} \id_{N})$.
Thus we have a well-defined morphism
\begin{align*}
\Psi_{i}: \Omega_{A}^{i} L \underline{\otimes}_{A} N  \rightarrow  \hTor_{i}^{A}(L, N)
\end{align*}
given by $w \mapsto [z]$.
It is easy to check that $\Psi_{i}$ is an isomorphism. 
\end{proof}
%
%

%
%
\begin{rem}  \label{rem:1}
{\rm
From Proposition $\ref{prop:9}$, we have isomorphisms
\begin{align*}
    \hHh^{i}(A, M) \cong \underline{\Hom}_{\Ae}(\Omega_{\Ae}^{i}A, M) \ \  \mbox{and}\ \  
    \hHh_{i}(A, M) \cong \Omega_{\Ae}^{i}A \underline{\otimes}_{\Ae} M.
\end{align*}
The vector spaces $\underline{\Hom}_{\Ae}(\Omega_{\Ae}^{i}A, M)$ and 
$\Omega_{\Ae}^{i}A \underline{\otimes}_{\Ae} M$  
are known to be the \textit{stable Hochschild cohomology and homology groups} defined by Eu and Shedler (see \cite{EuSched09}).
}
\end{rem}
%
%

The following lemma is the dual of \cite[Lemma 3.6]{BerghJorgen13}.
%
%
\begin{lemma} \label{lem:11}
Let $M$ and $N$ be $A$-modules and $\alpha \in \Aut(A)$. 
Then there exists an isomorphism
\[
\Ext_{A}^{i}(\,{}_{\alpha}M, N) \cong \Ext_{A}^{i}(M, {}_{\alpha^{-1}}N)
\] for all $i \geq 0$.
Moreover, if $A$ is a self-injective algebra, 
and $M$ is finitely generated, then there exists an isomorphism
\[
\hExt_{A}^{i}(\,{}_{\alpha}M, N) \cong \hExt_{A}^{i}(M, {}_{\alpha^{-1}}N)
\] for all $i \in \Z$.
\end{lemma}
\begin{proof}
Let $\P$ be a projective resolution of $M$.
It follows from the proof of \cite[Lemma 3.6]{BerghJorgen13} that ${}_{\alpha}\P$ is a projective resolution of ${}_{\alpha}M$.
Thus we have isomorphisms
\begin{align*}
    \Ext_{A}^{i}(\,{}_{\alpha}M, N) &\cong
    \Hh^{i}(\sHom_{A}( _{\alpha}\P, N)) \\ &\cong
    \Hh^{i}(\sHom_{A}(\P, {}_{\alpha^{-1}}N))  \cong
    \Ext_{A}^{i}(M, {}_{\alpha^{-1}}N).
\end{align*}
Assume now that $A$ is a self-injective algebra and that $M$ is finitely generated.
The proof of \cite[Lemma 3.6]{BerghJorgen13} implies that 
if $\T$ is a complete resolution of $M$, then $ _{\alpha}\T$ is a complete resolution of ${}_{\alpha}M$.
Replacing $\P$ by $\T$ in the argument above yields the later statement.
\end{proof}
%
%

For any automorphisms $\alpha, \beta \in \Aut(A)$, 
the \textit{twisted complex} ${}_{\alpha}C_{\beta}$ is defined to be
the chain complex with components $({}_{\alpha}C_{\beta})_{n} = {}_{\alpha}(C_{n})_{\beta}$ and the inherited differentials.

In the rest of this paper, we assume that all complete resolutions of a self-injective algebra $A$ are taken over $\Ae$. 
Using the two truncations $\T_{\geq 0}$ and $\T_{< 0}$, 
we can write $\T$ as $\T_{\geq 0} \xrightarrow{d_{0}^{\T}} \T_{< 0}$ with $d_{0}^{\T} = \eta \varepsilon$.
The quasi-isomorphism $A \xrightarrow{\eta} \T_{< 0}$ is called a \textit{backward projective resolution} of $A$.

%
%
\begin{lemma} \label{lem:2}
Let $A$ be a Frobenius algebra with the Nakayama automorphism $\nu$ and $\T$ arbitrary complete resolution of $A$. 
Then there exist two projective resolutions 
$\P \xrightarrow{\varepsilon} A$ and $\P^{\prime} \xrightarrow{\varepsilon^{\prime}} A$ such that
$\T$ is isomorphic to $\P \xrightarrow{d_{0}} {}_{1}((\boldsymbol{\Sigma}\P^{\prime})^{\vee})_{\nu}$, 
where $d_{0}$ is the composition of $\varepsilon$ with $\eta^{\prime}:= (\varepsilon^{\prime})^{\vee} t_{2}$.
\end{lemma}
\begin{proof}
It suffices to show that the backward projective resolution $A \xrightarrow{\eta} \T_{< 0}$ 
is obtained from some projective resolution in the desired way.

Take the $\Ae$-dual complex $(\boldsymbol{\Sigma} \T_{< 0})^{\vee}$ of the acyclic complex $\boldsymbol{\Sigma} \T_{< 0}$.
Twisting it by $\nu^{-1}$ on the left hand side,
we have a projective resolution
\[
\P^{\prime} := {}_{\nu^{-1}}((\boldsymbol{\Sigma}\T_{< 0})^{\vee})_{1} \xrightarrow{-\eta^{\vee}} {}_{\nu^{-1}}(A^{\vee})_{1} \cong A.
\]
Note that ${}_{\nu^{-1}}(T_{i}^{\vee})_{1}$ with $i \leq -1$ is a finitely generated projective $A$-bimodule.
Again, take the $\Ae$-dual complex 
$(\boldsymbol{\Sigma} \P^{\prime})^{\vee}$
of $\boldsymbol{\Sigma} \P^{\prime}$. 
Twisting it by $\nu$ on the right hand side, we get a backward projective resolution  
\begin{align*}
    A \xrightarrow{t_{2}} {}_{1}(A^{\vee})_{\nu} 
    \xrightarrow{\eta^{\vee \vee}} 
    {}_{1}((\boldsymbol{\Sigma} \P^{\prime})^{\vee})_{\nu}
    =
    {}_{1}(({}_{\nu^{-1}}((\T_{< 0})^{\vee})_{1})^{\vee})_{\nu}.
\end{align*}
This is isomorphic to $A \xrightarrow{\eta} \T_{< 0}$. 
Indeed, there are isomorphisms of complexes
\begin{align*}
{}_{1}(({}_{\nu^{-1}}((\T_{< 0})^{\vee})_{1})^{\vee})_{\nu} 
= &\ 
{}_{\id \otimes \nu}\sHom_{\Ae}(\sHom_{\Ae}(\T_{< 0}, \Ae)_{\id \otimes \nu^{-1}}, \Ae)\\
\cong &\ 
{}_{\id \otimes \nu}\sHom_{\Ae}(\sHom_{\Ae}(\T_{< 0}, \Ae), \Ae_{\id \otimes \nu})\\
\cong &\ 
{}_{\id \otimes \nu}\sHom_{\Ae}(\sHom_{\Ae}(\T_{< 0}, \Ae), {}_{\id \otimes \nu^{-1}}\Ae)\\
= &\ 
\sHom_{\Ae}(\sHom_{\Ae}(\T_{< 0}, \Ae), \Ae)\\
\cong &\ 
\T_{< 0},
\end{align*}
where the second isomorphism is induced by $\Ae$-bimodule isomorphism $\Ae_{\id \otimes \nu} \cong {}_{\id \otimes \nu^{-1}}\Ae$.
\end{proof}
%
%
We now recall the description of Tate-Hochschild (co)homology groups:
let $\T$ be a complete resolution of a Frobenius algebra $A$, $M$ an $A$-bimodule,
$C^{+} := \sHom_{\Ae}(\T_{\geq 0}, M)$ and $C^{-} := \sHom_{\Ae}(\T_{< 0}, M)$.
Since $\T_{\geq 0}$ is a projective resolution of $A$, we have $\Hh^{i}(C^{+}) = \Hh^{i}(A, M)$ for $i \geq 0$.
It follows from  Lemma \ref{lem:2} that $\T_{< 0} \cong {}_{1}((\boldsymbol{\Sigma} \P)^{\vee})_{\nu}$ for some projective resolution $\P$ of $A$.
Thus, we have
\begin{align*}
    C^{-} &= \sHom_{\Ae}(\T_{< 0}, M) 
            \cong \sHom_{\Ae}({}_{1}((\boldsymbol{\Sigma} \P)^{\vee})_{\nu}, M) \\[3pt]
            &\cong \sHom_{\Ae}((\boldsymbol{\Sigma} \P)^{\vee}, {}_{1}M_{\nu^{-1}})
            \cong \boldsymbol{\Sigma} \P \otimes_{\Ae} {}_{1}M_{\nu^{-1}}
\end{align*}
and hence $\Hh^{i}(C^{-}) \cong \Hh_{-i-1}(A, {}_{1}M_{\nu^{-1}})$ for $i \leq -1$.
Therefore, we get
\begin{align*}
    \hHh^{n}(A, M)
    =\begin{cases}
        \Hh^{n}(A, M) & \mbox{ if } n > 0, \\
        \Hh_{-n-1}(A, {}_{1}M_{\nu^{-1}}) & \mbox{ if } n < -1.
    \end{cases}
\end{align*}
For the other two cohomology groups, there are isomorphisms
\begin{align} \label{eq:5}
\hHh^{-1}(A, M) \cong ({}_{N_{A}}M)/ I_{A}(M), \qquad \hHh^{0}(A, M) \cong M^{A}/N_{A}(M),
\end{align}
where we set
		\begin{align*} \lefteqn{}
			& M^{A} := \{ \, m \in M \ | \ a m = m a \mbox{ for all } a \in A \}, \\[5pt]
			&N_{A}(M) := \bigg \{ \,  \sum_{i} u_{i} m v_{i} \ \bigg| \ m \in M \bigg \}, \quad
			{}_{N_{A}}M := \left \{ \, m \in M \ \bigg| \ \sum_{i} u_{i} m v_{i} = 0 \right \}, \\
			&I_{A}(M) :=  \left \{ \, \sum_{j}(m_{i} \nu^{-1}(a_{i}) -a_{i} m_{i}) \ \bigg| \ m_{i} \in M, a_{i} \in A \right\}.
		\end{align*}

The vector spaces $({}_{N_{A}}M)/ I_{A}(M)$ and $M^{A}/N_{A}(M)$ appear in the following exact sequence
\begin{align} \label{eq:3}
    0 \rightarrow 
    {}_{N_{A}}M/ I_{A}(M) \rightarrow 
    \Hh_{0}(A, {}_{1}M_{\nu^{-1}}) \xrightarrow{\overline{N}} 
    \Hh^{0}(A, M) \rightarrow 
    M^{A}/N_{A}(M) \rightarrow 
    0,
\end{align}
where $\overline{N}$ is induced by the norm map $N: {}_{1}M_{\nu^{-1}} \rightarrow M$ defined in \cite{Nakayama57,Sanada92} which sends $m \in {}_{1}M_{\nu^{-1}}$ to $N(m) = \sum_{i} u_{i} m v_{i}$.
In order to prove the existence of the exact sequence $(\ref{eq:3})$, without loss of generality, we may suppose that 
the beginning of $\T$ is of the form
\begin{align*} 
\xymatrix@C=15pt@R=10pt@M=4pt{
\cdots \rightarrow
T_{2} \ar[r] &
A^{\otimes 3} \ar[r]^{d_{1}} &
A^{\otimes 2} \ar[rr]^{d_{0}} \ar[rd]_{\varepsilon} &
&
{}_{1}(A^{\otimes 2})_{\nu} \ar[r]^-{d_{-1}}  &
{}_{1}((A^{\otimes 3})^{\vee})_{\nu} \rightarrow
\cdots \\
&
&
&
A \ar[ru]_{\eta} &
&
}
\end{align*}
where the maps above are given by as follows: 
\begin{align*}
    &d_{1}: A^{\otimes 3} \rightarrow A^{\otimes 2}; \quad a \otimes b \otimes c \mapsto a b \otimes c -a \otimes bc, \\
    &\varepsilon: A^{\otimes 2} \rightarrow A; \quad x \otimes y \mapsto xy, \\
    &\eta : A \rightarrow {}_{1}(A^{\otimes 2})_{\nu}; \quad x \mapsto \sum_{i} u_{i} \nu(x) \otimes v_{i}, \\
    &d_{-1}: {}_{1}(A^{\otimes 2})_{\nu} \rightarrow {}_{1}((A^{\otimes 3})^{\vee})_{\nu}; 
    \quad x \otimes y \mapsto [1 \otimes a \otimes 1 \mapsto ax \otimes y -x \otimes ya].
\end{align*}
A direct calculation shows that there exists a morphism $\overline{N}: \Hh_{0}(A, {}_{1}M_{\nu^{-1}}) \rightarrow \Hh^{0}(A, M)$ making the following square
commute:
\[\xymatrix@C=19pt@R=19pt@M=8pt{
{}_{1}M_{\nu^{-1}} \ar[r]^-{N} \ar@{->>}[d] &
M \\
\Hh_{0}(A, {}_{1}M_{\nu^{-1}}) \ar[r]^-{\overline{N}} &
\Hh^{0}(A, M) \ar@{>->}[u] 
}\]
Since $N :{}_{1}M_{\nu^{-1}} \rightarrow M$ is the $0$-th differential of $\sHom_{\Ae}(\T, M)$, we get the isomorphisms (\ref{eq:5}) and the exact sequence (\ref{eq:3}).

Using Lemma $\ref{lem:11}$, one analogously checks that the Tate-Hochschild homology groups $\hHh_{*}(A, M)$ can be written as follows: 
\begin{align*}
    \hHh_{n}(A, M)
    =\begin{cases}
        \Hh_{n}(A, M) & \mbox{ if } n > 0, \\[3pt]
        \Hh^{-n-1}(A, {}_{1}M_{\nu}) & \mbox{ if } n < -1,    \end{cases}
\end{align*}
and there exists an exact sequence
\begin{align*}
    0 \rightarrow \hHh_{0}(A, M) \rightarrow \Hh_{0}(A, M) \xrightarrow{\overline{N^{\prime}}} \Hh^{0}(A, {}_{1}M_{\nu}) \rightarrow \hHh_{-1}(A, M) \rightarrow 0,
\end{align*} 
where $\overline{N^{\prime}}$ is induced by the morphism $N^{\prime}: M \rightarrow {}_{1}M_{\nu}$ sending $m \in M$ to 
\begin{align*}
   \sum_{i} u_{i} m \nu(v_{i}) \in {}_{1}M_{\nu}.
\end{align*}
Therefore, we have an isomorphisms 
\begin{align*}
    \hHh^{n}(A, M) \cong \hHh_{-n-1}(A, {}_{1}M_{\nu^{-1}})
\end{align*} 
for $n \in \Z$ and an $A$-bimodule $M$.

We end this section by recalling the definition of weakly projective bimodules in the sense of Sanada \cite{Sanada92}.
For bimodules $M$ and $N$ over a Frobenius algebra $A$, the space $\Hom_{-,A}(M, N)$ of morphisms of right $A$-modules becomes an $A$-bimodule by defining
\[
(a g b)(m) := a g( b m )
\]
for $a, b \in A, g \in \Hom_{-,A}(M, N)$  and $m \in M$. 
Similarly, the space $\Hom_{A,-}(M, N)$ of morphisms of left $A$-modules becomes an $A$-bimodule by defining
\[
(a g b)(m) := g( m a ) b
\]
for $a, b \in A, g \in \Hom_{A,-}(M, N)$ and $m \in M$. 
Then we see that 
\[
\sum_{i} u_{i} g v_{i} \in \Hom_{\Ae}(M, N)
\]
for all $g \in \Hom_{-,A}(M, N)$, or $g \in \Hom_{A,-}(M, N)$,
where $\{u_{i}\}_{i}$ and $\{v_{i}\}_{i}$ are dual bases of $A$.
We say that an $A$-bimoudle $M$ is \textit{weakly projective}  if
there exists either $g \in \Hom_{-,A}(M, M)$ or $g \in \Hom_{A,-}(M, M)$ such that 
\[
\sum_{i} u_{i} g v_{i} = \id_{M} \in \End_{\Ae}(M).
\]
For an $A$-bimodule $M$, Sanada provided four exact sequences of $A$-bimodules 
which split as exact sequences of one-sided $A$-modules and whose middle terms are weakly projective:
\begin{align*}
    &0 \rightarrow K(M) \rightarrow A \otimes M \rightarrow M \rightarrow 0 \quad \mbox{(right $A$-splitting),}\\
    &0 \rightarrow K^{\prime}(M) \rightarrow M \otimes A \rightarrow M \rightarrow 0 \quad \mbox{(left $A$-splitting),}\\
    &0 \rightarrow M \rightarrow \Hom_{k}(A_{A}, M_{A}) \rightarrow C(M) \rightarrow 0 \quad \mbox{(right $A$-splitting),}\\
    &0 \rightarrow M \rightarrow \Hom_{k}({}_{A}A, {}_{A}M) \rightarrow C^{\prime}(M) \rightarrow 0 \quad \mbox{(left $A$-splitting).}
\end{align*}

%
%
\begin{lemma}[{\cite[Lemma 1.3]{Sanada92}}] \label{lem:6}
Let $A$ be a Frobenius algebra.
If an $A$-bimodule $M$ is weakly projective, then $\hHh^{i}(A, M)$ vanishes for all $i \in \Z$.
\end{lemma}
%
%

%
%
\begin{cor} \label{cor:1}
Let $A$ be a Frobenius algebra.
If an $A$-bimodule $M$ is weakly projective, then $\hHh_{i}(A, M)$ vanishes for all $i \in \Z$.
\end{cor}
\begin{proof}
Let $M$ be a weakly projective $A$-bimodule.
Observe that ${}_{\alpha}M_{\beta}$ is also weakly projective for any $\alpha, \beta \in \Aut(A)$.
Therefore, the statement follows from the fact that there exists an isomorphism 
\[
\hHh_{i}(A, M) \cong \hHh^{-i-1}(A, {}_{1}M_{\nu})
\]
for any $i \in \Z$, where $\nu$ is the Nakayama automorphism of $A$.
\end{proof}
%
%

Each of the four exact sequences above yields a long exact sequence of Tate-Hochschild (co)homology groups with connecting homomorphisms $\partial$, 
so that we have the following.
%
%
\begin{cor}[{\cite[Corollary 1.5]{Sanada92}}] \label{cor:2}
for any $i \in \Z$, there exist isomorphisms 
\begin{align*}
    &\partial: \hHh^{i}(A, M) \xrightarrow{\ \sim \ } \hHh^{i+1}(A, K(M)) \quad &(or \xrightarrow{\ \sim \ } \hHh^{i+1}(A, K^{\prime}(M))), \\[5pt]
    &\partial^{-1} :\hHh^{i}(A, M) \xrightarrow{\ \sim \ } \hHh^{i-1}(A, C(M)) \quad &(or \xrightarrow{\ \sim \ } \hHh^{i-1}(A, C^{\prime}(M))), \\[5pt]
    &\partial: \hHh_{i}(A, M) \xrightarrow{\ \sim \ } \hHh_{i-1}(A, K(M)) \quad &(or \xrightarrow{\ \sim \ } \hHh_{i-1}(A, K^{\prime}(M))), \\[5pt]
    &\partial^{-1} :\hHh_{i}(A, M) \xrightarrow{\ \sim \ } \hHh_{i+1}(A, C(M)) \quad &(or \xrightarrow{\ \sim \ } \hHh_{i+1}(A, C^{\prime}(M))).
\end{align*}
\end{cor}
%
%

%
%
\section{Cup product and cap product in Tate-Hochschild theory}
\label{sec:3}
Our aim in this section is to prove our main theorem. 
For this purpose, we define cup product  and cap product on Tate-Hochschild (co)homology
in an analogous way to the discussion in \cite{Brown94}.
As in Hochschild theory, we also give two certain products, 
called \textit{composition products}, which are equivalent to the cup product and the cap product, respectively.
Throughout this section, let $A$ denote a finite dimensional Frobenius algebra 
over a field $k$, unless otherwise stated. 
%
%
\subsection{Cup product and cap product}
\label{sebsec:1}
In this subsection, we show the existence of diagonal approximation for arbitrary complete resolution of $A$ and prove that all diagonal approximations define exactly one cup product and one cap product.  

Recall that the cup product on Hochschild cohomology groups is defined by using a diagonal approximation $\Delta: \P \rightarrow \P \otimes_{A} \P$ 
for a single projective resolution $\P$ of $A$.
We will define cup product on Tate-Hochschild cohomology groups in a similar way. 
However, we fail to develop the theory of cup product
when we use the ordinary tensor product, because we need to define the cup product of two elements of degree $i$ and of degree $j$ 
with $i, j \in \Z$.
For this, we must consider the \textit{complete tensor product} of two chain complexes.
%
%
\begin{defi}
Let $C$ be a chain complex of right modules over a $($not necessarily Frobenius$)$ algebra $A$ and $C^{\prime}$ a chain complex of left $A$-modules. 
Then the \textit{complete tensor product} $C \hotimes_{A} C^{\prime}$ is defined to be the chain complex with components
\[
(C \hotimes_{A} C^{\prime})_{n} := \prod_{i \in \Z} C_{n-i} \otimes_{A} C^{\prime}_{i}
\]
and differentials
\[
d_{n}^{C \hotimes_{A} C^{\prime}}\left((x_{n-i} \otimes_{A} x^{\prime}_{i})_{i\in \Z} \right) 
:= 
\left(
d^{C}(x_{n-i}) \otimes_{A} x^{\prime}_{i} +(-1)^{n-i} x_{n-i} \otimes_{A} d^{C^{\prime}}(x^{\prime}_{i}) 
\right)_{i\in \Z}
\] 
for any $(x_{n-i} \otimes_{A} x^{\prime}_{i})_{i \in \Z} \in (C \hotimes_{A} C^{\prime})_{n}$.
\end{defi}
%
%

%
%
\begin{defi}
Let $C_{1}$ and $C^{\prime}_{1}$ be chain complexes of right modules over a $($not necessarily Frobenius$)$ algebra $A$ and
$C_{2}$ and $C^{\prime}_{2}$ chain complexes of left $A$-modules.
Let $u: C_{1} \rightarrow C^{\prime}_{1}$ and $v: C_{2} \rightarrow C^{\prime}_{2}$ be 
graded $A$-linear maps of degree $s$ and of degree $t$.
The \textit{complete tensor product} $u \hotimes_{A} v$ of $u$ with $v$ over $A$ is defined as the $A$-linear graded map of degree $s+t$ defined by 
\[
(u \hotimes_{A} v) \left((x_{n-i} \otimes y_{i})_{i \in \Z} \right) 
:= \left((-1)^{t(n-i)} u(x_{n-i}) \otimes_{A} v(y_{i}) \right)_{i \in \Z}
\]
for any $(x_{n-i} \otimes y_{i})_{i \in \Z} \in (C_{1} \hotimes_{A} C_{2})_{n}$.
The complete tensor product of graded maps induces a chain map
\[
\sHom_{A}(C_{1}, C^{\prime}_{1}) \otimes \sHom_{A}(C_{2}, C^{\prime}_{2}) \rightarrow \sHom_{A}(C_{1} \hotimes_{A} C_{2}, C^{\prime}_{1} \hotimes_{A} C^{\prime}_{2}).
\]
\end{defi}
%
%
\noindent
Remark that we can rewrite the differential $d^{C \hotimes_{A} C^{\prime}}$ defined above as 
\[
d^{C \hotimes_{A} C^{\prime}} = d^{C} \hotimes_{A} \id + \id \hotimes_{A} d^{C^{\prime}}.
\]

Let $\T$ be a complete resolution of $A$
with augmentation $\varepsilon: \T \rightarrow A$. 
We say that a chain map $\widehat{\Delta} : \T \rightarrow \T \hotimes_{A} \T$ is a \textit{diagonal approximation} if it satisfies 
$(\varepsilon \hotimes_{A} \varepsilon) \widehat{\Delta} = \varepsilon$.
Note that $\T \hotimes_{A} \T$ is no longer a complete resolution of $A$, 
because all the components of $\T \hotimes_{A} \T$ are not finitely generated.
Nevertheless, we prove that there exists a diagonal approximation for arbitrary complete resolution of $A$. For this purpose, we need the following three lemmas.

%
%
\begin{lemma} \label{lem:3}

Let $\T$ be a complete resolution of $A$. 
Then $\T \hotimes_{A} \T$ is acyclic and dimensionwise projective as $A$-bimodules. 
\end{lemma}
\begin{proof}
First, we will show the projectivity of each component of $\T \hotimes_{A} \T$.
This follows from our assumption that $A$ is Frobenius and the fact that $\Ae \otimes_{A} \Ae$ is projective over $\Ae$.
In order to prove the acyclicity of $\T \hotimes_{A} \T$, we construct a contracting homotopy for $\T \hotimes_{A} \T$ as a complex of right $A$-modules.
Since any complete resolution $\T$ of $A$ is contractible as a complex of left $A$-modules, 
we obtain

a contracting homotopy $h$.
Set $H = \id_{\T} \hotimes_{A} h : \T \hotimes_{A} \T \rightarrow \T \hotimes_{A} \T$.
Then $H$ is a contracting homotopy for $\T \hotimes_{A} \T$.
Indeed, we have 
\begin{align*}
    d^{\T \hotimes_{A} \T} H &+H d^{\T \hotimes_{A} \T} \\
    &=
    (d^{\T} \hotimes_{A} \id + \id \hotimes d^{\T})(\id \hotimes h) 
    + (\id \hotimes h)(d^{\T} \hotimes_{A} \id + \id \hotimes d^{\T})\\
    &=
    d^{\T} \hotimes_{A} h + \id \hotimes_{A} d^{\T} h -d^{\T} \hotimes_{A} h + \id \hotimes_{A} h d^{\T} \\
    &=
    \id \hotimes_{A} (d^{\T} h +h d^{\T})
    =
    \id \hotimes_{A} \id.
\end{align*}
This completes the proof.
\end{proof}
%
%

%
%
\begin{lemma} \label{lem:4}
Let $C$ and $C^{\prime}$ be acyclic chain complexes over 
a $($not necessarily Frobenius$)$ algebra $A$.
Assume that $C_{i}$ is projective over $A$ for $i \geq 1$ and that $C^{\prime}_{i}$ is injective over $A$ for $i \leq -1$.
If there exists a morphism $\tau_{0}: C_{0} \rightarrow C^{\prime}_{0}$ satisfies $d^{C^{\prime}}_{0} \tau_{0} d^{C}_{1} = 0$, 
then $\tau_{0}$ extends to a chain map $\tau: C \rightarrow C^{\prime}$, up to homotopy. 
\end{lemma}
\begin{proof}
Consider a commutative diagram
\[\xymatrix{
\cdots \ar[r] &
C_{2} \ar[r]^-{d^{C}_{2}} &
C_{1} \ar[r]^-{d^{C}_{1}} &
C_{0} \ar[r]^-{d^{C}_{0}} \ar[d]^-{\tau_{0}} &
\Coker d^{C}_{1}  \ar[r] \ar[d]^-{\overline{d^{C^{\prime}}_{0} \tau_{0}}} &
0  \ar[r] \ar[d] &
\cdots \\
\cdots \ar[r] &
C^{\prime}_{2} \ar[r]^-{d^{C^{\prime}}_{2}} &
C^{\prime}_{1} \ar[r]^-{d^{C^{\prime}}_{1}} &
C^{\prime}_{0} \ar[r]^-{d^{C^{\prime}}_{0}} &
C^{\prime}_{-1} \ar[r] &
0  \ar[r] &
\cdots
}\]
where the morphism $\overline{d^{C^{\prime}}_{0} \tau_{0}}$ is given by 
$\overline{x_{0}} \mapsto d^{C^{\prime}}_{0} \tau_{0}(x_{0})$
for $\overline{x_{0}} \in \Coker d^{C}_{1}$.
It follows form Lemma $\ref{lem:7}$ that 
there uniquely (up to homotopy) exists a family $\{\tau_{i}: C_{i} \rightarrow C^{\prime}_{i}\}_{i \geq 0}$
such that $d^{C^{\prime}}_{i} \tau_{i} = \tau_{i-1} d^{C}_{i}$ for $i \geq 1$. 
Applying Lemma \ref{lem:8} to a commutative diagram
\[\xymatrix{
\cdots \ar[r] &
C_{2} \ar[r]^-{d^{C}_{2}}\ar[d]^-{\tau_{2}} &
C_{1} \ar[r]^-{d^{C}_{1}} \ar[d]^-{\tau_{1}}&
C_{0} \ar[r]^-{d^{C}_{0}} \ar[d]^-{\tau_{0}} &
C_{-1} \ar[r]^-{d^{C}_{-1}} &
C_{-2}  \ar[r]  &
\cdots \\
\cdots \ar[r] &
C^{\prime}_{2} \ar[r]^-{d^{C^{\prime}}_{2}} &
C^{\prime}_{1} \ar[r]^-{d^{C^{\prime}}_{1}} &
C^{\prime}_{0} \ar[r]^-{d^{C^{\prime}}_{0}} &
C^{\prime}_{-1} \ar[r]^-{d^{C^{\prime}}_{-1}} &
C^{\prime}_{-2}  \ar[r] &
\cdots
}\]
we have that the family $\{\tau_{i}\}_{i \geq 0}$ extends to a chain map $\tau :C \rightarrow C^{\prime}$, 
which is uniquely determined up to homotopy.
\end{proof}
%
%

%
%
\begin{lemma} \label{lem:5}
Let $P$ be a finitely generated projective bimoudle over a $($not necessarily Frobenius$)$ algebra $A$ and 
$C$ an acyclic chain complex of $A$-bimodules.
Assume that $C$ is contractible as a complex of left $($resp., right$)$ $A$-modules.
Then $C \otimes_{A} P\ ($resp., $P \otimes_{A} C )$ is contractible as a complex of $A$-bimodules.
\end{lemma}
\begin{proof}
We prove the statement only for the case that $C$ is contractible as a complex of left $A$-modules.
It suffices to show the statement for $P = \Ae$.
We construct a contracting homotopy for $C \otimes_{A} \Ae$. 
Let $h$ be a contracting homotopy for $_{A}C$. 
A direct computation shows that the graded map $h \otimes \id_{A}: C \otimes A \rightarrow C \otimes A$
of $A$-bimodules of degree $1$ is a contracting homotopy. 
Since there is an isomorphism $C \otimes_{A} \Ae \cong C \otimes A$ of chain complexes of $A$-bimodules, 
the graded map induced by $h \otimes \id_{A}$ via this isomorphism is
a contracting homotopy for $C \otimes_{A} \Ae$.
\end{proof}
%
%

We are now able to show the existence of a diagonal approximation for any complete resolution of a Frobenius algebra.
%
%
\begin{theo} \label{theo:1}
Let $\T$ be a complete resolution of $A$.
Then there uniquely $($up to homotopy$)$ exists a diagonal approximation $\widehat{\Delta}: \T \rightarrow \T \hotimes_{A} \T$. 
\end{theo}
\begin{proof}
In view of Lemmas \ref{lem:3} and \ref{lem:4}, it suffices to construct a map 
$\tau: T_{0} \rightarrow \prod_{i \in \Z} (T_{i} \otimes_{A} T_{-i})$ such that
$(1)\,d^{\T \hotimes_{A} \T}_{0} \,\tau\, d^{\T}_{1} = 0$ and $(2)\,(\varepsilon \hotimes_{A} \varepsilon) \tau = \varepsilon$,
where $\varepsilon : \T \rightarrow A$ is the augmentation of $\T$.
Let $\tau_{r}:  T_{0} \rightarrow T_{r} \otimes_{A} T_{-r}$ denote the composition of $\tau$ with the $r$-th canonical projection on
$\prod_{i \in \Z} (T_{i} \otimes_{A} T_{-i})$.
Set 
\[d^{\prime}_{i,j} := d^{\T}_{i} \otimes_{A} \id_{T_{j}} \quad \mbox{and} \quad d^{\prime \prime}_{i,j} := (-1)^{i} \id_{T_{j}} \otimes_{A} d^{\T}_{i}.
\]
Then we can rewrite the first condition $(1)$ as 
\begin{align} \label{eq:4}
\left( d^{\prime}_{i,-i} \tau_{i} + d^{\prime \prime}_{i-1,-i+1} \tau_{i-1} \right)\big|_{B_{0}(\T)}= 0 \quad
\mbox{ for each $i \in \Z$.}
\end{align}
Since $T_{0}$ is projective, there exists a morphism $\tau_{0}: T_{0} \rightarrow T_{0} \otimes_{A} T_{0}$ of $A$-bimodules such that
$(\varepsilon \hotimes_{A} \varepsilon) \tau_{0} = \varepsilon$.
Suppose that $i > 0$ and that we have defined $\tau_{j}$ with $0 \leq j < i$ satisfying the condition (\ref{eq:4}).
Consider the following diagram with exact row:
\[\xymatrix@R=40pt@C=50pt{
&
B_{0}(\T) \ar[d]^-{-d^{\prime \prime}_{i-1,-i+1}  \tau_{i-1}\big|_{B_{0}(\T)}}&
 \\
T_{i} \otimes_{A} T_{-i} \ar[r]^-{d^{\prime}_{i,-i}} &
T_{i-1} \otimes_{A} T_{-i} \ar[r]^-{d^{\prime}_{i-1,-i}} &
T_{i-2} \otimes_{A} T_{-i}
}\]
It follows from Lemma \ref{lem:5} that
the complex $(\T \otimes_{A} T_{-i}, d^{\prime}_{\bullet,-i})$ of $A$-bimodules is contractible.
Let $h =( h_{i} )_{i\in \Z}$ be a contracting homotopy for $\T \otimes_{A} T_{-i}$.
Put 
\[
\tau_{i} := -h_{i-1} (d^{\prime \prime}_{i-1,-i+1} \tau_{i-1}): T_{0} \rightarrow T_{i} \otimes_{A} T_{-i}.
\]

We now claim that $d^{\prime}_{i-1,-i}  (-d^{\prime \prime}_{i-1,1-i}  \tau_{i-1})\big|_{B_{0}(\T)}= 0$ hols for $i \geq 1$.
If $i =1$, then we have on $B_{0}(\T)$
\begin{align*}
    d^{\prime}_{0,-1}  (-d^{\prime \prime}_{0,0} \tau_{0})
    &=
    -(d_{0} \otimes_{A} d_{0}) \tau_{0} \\
    &=
    -(\eta \otimes_{A} \eta)(\varepsilon \otimes_{A} \varepsilon) \tau_{0}  \\
    &= 
    0.
\end{align*}
Assume that $i>1$. We know that 
\[
(d^{\prime}_{i-1,1-i}  \tau_{i-1} +d^{\prime \prime}_{i-2,2-i}  \tau_{i-2})\big|_{B_{0}(\T)} = 0
\] 
holds, so that we have  on $B_{0}(\T)$
\begin{align*}
    d^{\prime}_{i-1,-i}  (-d^{\prime \prime}_{i-1,1-i} \tau_{i-1}) 
    &=
     (-d^{\prime}_{i-1,-i}  d^{\prime \prime}_{i-1,1-i}) \tau_{i-1}  \\
    &=
     (d^{\prime \prime}_{i-2,1-i}  d^{\prime}_{i-1,1-i}) \tau_{i-1}  \\
    &=
    -d^{\prime \prime}_{i-2,1-i}  (d^{\prime \prime}_{i-2,2-i} \tau_{i-2})  \\
    &= 0
\end{align*}
and hence on $B_{0}(\T)$
\begin{align*}
    d^{\prime}_{i,-i}  \tau_{i} 
    &=
    d^{\prime}_{i,-i}  h_{i-1}  (-d^{\prime \prime}_{i-1,-i+1}  \tau_{i-1}) \\
    &=
    (\id -h_{i-2}  d^{\prime}_{i-1,-i}) (-d^{\prime \prime}_{i-1,-i+1}  \tau_{i-1}) \\
    &=
     (-d^{\prime \prime}_{i-1,-i+1}  \tau_{i-1}) +h_{i-2}  (d^{\prime}_{i-1,-i} (d^{\prime \prime}_{i-1,-i+1}  \tau_{i-1})) \\
    &= 
    -d^{\prime \prime}_{i-1,-i+1}  \tau_{i-1}. 
\end{align*}
We have constructed $\{\tau_{i}\}_{i \geq 0}$ satisfying the condition (\ref{eq:4}). 
A dual argument using descending induction on $i < 0$ shows that we get the other components $\tau_{i}$ with $i < 0$.
In conclusion, the statement follows from Lemmas \ref{lem:3} and \ref{lem:4}.
\end{proof}
%
%

%
%
\begin{rem}
{\rm
Let $\T$ be a complete resolution of $A$. 
We denote by $\P$ the non-negative truncation $\T_{\geq 0}$.
Recall that $\P$ is a projective resolution of $A$.
Set 
\begin{align*}
    \T \hotimes_{A} \T_{\geq 0} 
    :=
    \bigoplus_{n \geq 0} \left( \bigoplus_{0 \leq p \leq n} T_{n-p} \otimes_{A} T_{p} \right)
\end{align*}
and 
\begin{align*}
    d^{\T \hotimes_{A} \T_{\geq 0}}
    := d_{\geq 0}^{\T} \otimes_{A} \id + \id \otimes_{A} d_{\geq 0}^{\T}.
\end{align*}
Then $(\T \hotimes_{A} \T_{\geq 0}, d^{\T \hotimes_{A} \T_{\geq 0}})$ is nothing but a projective resolution $\P \otimes_{A} \P$ of $A$.
If $\widehat{\Delta}: \T \rightarrow \T \hotimes_{A} \T$ is a diagonal approximation, then it can be decomposed as
\begin{align*}
    \widehat{\Delta} =\prod_{n \in \Z}\left(\, \prod_{p \in \Z} \widehat{\Delta}^{(n)}_{p} \right) 
\end{align*}
with $\widehat{\Delta}^{(n)}_{p}: T_{n} \rightarrow T_{n-p} \otimes_{A} T_{p}$. 
We denote by $\widehat{\Delta}_{\geq 0}$ a graded $\Ae$-linear map 
\begin{align*}
   \prod_{n \geq 0}\left( \prod_{0 \leq p \leq n} \widehat{\Delta}^{(n)}_{p} \right): \T_{\geq 0} \rightarrow \T \hotimes_{A} \T_{\geq 0}.
\end{align*}
By the definition of $\widehat{\Delta}_{\geq 0}$, 
the graded map $\widehat{\Delta}_{\geq 0}$ becomes an augmentation-preserving chain map from $\P$ to $\P \otimes_{A} \P$,
which means that $\widehat{\Delta}_{\geq 0}$ is a diagonal approximation for the projective resolution $\P$.
}
\end{rem}
%
%

Theorem \ref{theo:1} allows us to define cup product and cap product on Tate-Hochschild (co)homology groups:
for $A$-bimodules $M$ and $N$, we define a graded $k$-linear map
\begin{align} \label{eq:6}
\smile : \sHom_{\Ae}(\T, M) \otimes \sHom_{\Ae}(\T, N) \rightarrow \sHom_{\Ae}(\T, M \otimes_{A} N)
\end{align}
by $u \smile v := (u \hotimes_{A} v) \widehat{\Delta}$ for homogeneous elements $u \in \sHom_{\Ae}(\T, M)$ and $v \in \sHom_{\Ae}(\T, N)$.
One can easily check that $\smile$ is a chain map, so that it induces an operator
\begin{align} \label{eq:8}
\smile: \hHh^{*}(A, M) \otimes \hHh^{*}(A, N) \rightarrow \hHh^{*}(A, M \otimes_{A} N).
\end{align}
For $u \in \hHh^{r}(A, M)$ and $v \in \hHh^{s}(A, N)$, we call $u \smile v\in \hHh^{r+s}(A, M \otimes_{A} N)$ the \textit{cup product} of $u$ and $v$.
On the other hand, consider the composition
\begin{align*}
    \frown: \sHom_{\Ae}(\T, M) \otimes (\T \otimes_{\Ae} N) \rightarrow \T \otimes_{\Ae}(M \otimes_{A} N)
\end{align*}
of two chain maps
\begin{align*}
    \id &\otimes(\widehat{\Delta} \otimes_{\Ae} \id): \\
    &\sHom_{\Ae}(\T, M) \otimes (\T \otimes_{\Ae} N) \rightarrow \sHom_{\Ae}(\T, M) \otimes ((\T \hotimes_{A} \T) \otimes_{\Ae} N)
\end{align*}
and
\begin{align*}
    \gamma: \sHom_{\Ae}(\T, M) \otimes ((\T \hotimes_{A} \T) \otimes_{\Ae} N) \rightarrow \T \otimes_{\Ae}(M \otimes_{A} N)
\end{align*}
given by 
\begin{align*}
    u \otimes((x \hotimes_{A} y) \otimes_{\Ae} n) \mapsto (-1)^{|u||x|} x \otimes_{\Ae}( u(y) \otimes_{A} n).
\end{align*}
Then the composition $\frown$ induces an operator
\begin{align} \label{eq:10}
\frown: \hHh^{r}(A, M) \otimes \hHh_{s}(A, N) \rightarrow \hHh_{s-r}(A, M \otimes_{A} N).
\end{align}
For $u \in \hHh^{r}(A, M)$ and $z \in \hHh_{s}(A, N)$, we call $u \frown z \in \hHh_{s-r}(A, M \otimes_{A} N)$ the \textit{cap product} of $u$ and $z$.

Now, we will show the uniqueness of the cup product $\smile$ and of the cap product $\frown$, that is, 
each of them does not depend on the choice of a complete resolution and a diagonal approximation.
First, we will deal with the cup product.
%
%
\begin{prop} \label{prop:1}
The cup product $\smile$ satisfies the following three properties.
\begin{enumerate}
\item[(PI)]
    Let $M$ and $N$ be $A$-bimodules. Then there exists a commutative square
    \[
    \xymatrix{
    \hHh^{0}(A, M) \otimes \hHh^{0}(A, N) \ar[r]^{\smile} \ar[d]&
    \hHh^{0}(A, M \otimes_{A} N) \ar[d] \\
    M^{A}/N_{A}(M) \otimes N^{A}/N_{A}(N) \ar[r]   &
    (M \otimes_{A} N)^{A}/N_{A}(M \otimes_{A} N)
    }\]
    where the vertical morphisms are isomorphisms in $(\ref{eq:5})$ and the morphism in the bottom row is given by
    \[
    (m + N_{A}(M)) \otimes (n + N_{A}(N)) \mapsto m \otimes_{A} n + N_{A}(M \otimes_{A} N).
    \]
\item[(PII$_{1}$)]
    Let 
    \begin{align*}
        &0 \rightarrow M_{1} \rightarrow M_{2} \rightarrow M_{3} \rightarrow 0, \\
        &0 \rightarrow M_{1} \otimes_{A} N \rightarrow M_{2} \otimes_{A} N \rightarrow M_{3} \otimes_{A} N \rightarrow 0
    \end{align*}
    be exact sequences of $A$-bimodules. Then we have 
    \[
    \partial(\gamma \smile \xi) = (\partial \gamma) \smile \xi
    \]
    for all $\gamma \in \hHh^{r}(A, M_{3})$ and $\xi \in \hHh^{s}(A, N)$, 
    where $\partial$ denotes the connecting homomorphism.
\item[(PII$_{2}$)]
    Let 
    \begin{align*}
        &0 \rightarrow N_{1} \rightarrow N_{2} \rightarrow N_{3} \rightarrow 0, \\
        &0 \rightarrow M \otimes_{A} N_{1} \rightarrow M \otimes_{A} N_{2} \rightarrow M \otimes_{A} N_{3} \rightarrow 0
    \end{align*}
    be exact sequences of $A$-bimodules. Then we have 
    \[
    \partial(\gamma \smile \xi) = (-1)^{r}  \gamma \smile (\partial \xi )
    \]
    for all $\gamma \in \hHh^{r}(A, M)$ and $\xi \in \hHh^{s}(A, N_{3})$, 
    where $\partial$ denotes the connecting homomorphism.
\end{enumerate}
\end{prop} 
\begin{proof}
We only prove that the cup product satisfies (PI) and (PII$_{1}$);
the proof of (PII$_{2}$) is similar to (PII$_{1}$).
Let $\T$ be a complete resolution of $A$ and $\widehat{\Delta}: \T \rightarrow \T \hotimes_{A} \T$ a diagonal approximation.
The property that $(\varepsilon \hotimes_{A} \varepsilon) \widehat{\Delta} = \varepsilon$ implies that the cup product satisfies (PI).

For (PII$_{1}$), let $\alpha \in \sHom_{\Ae}(\T, M)^{s}$ be a cocycle representing $\xi \in \hHh^{s}(A, N)$.
Then the graded map 
\[
-\smile \alpha: \sHom_{\Ae}(\T, M_{1}) \rightarrow \sHom_{\Ae}(\T, M_{1} \otimes_{A} N)
\]
induced by $\alpha$ is a chain map. 
Thus, we have a commutative diagram of chain complexes with exact rows \\[3pt]
\scalebox{0.87}{
\xymatrix@R=25pt@C=17pt{
0 \ar[r] &
\sHom_{\Ae}(\T, M_{1}) \ar[r] \ar[d]^-{-\smile \alpha} &
\sHom_{\Ae}(\T, M_{2}) \ar[r] \ar[d]^-{-\smile \alpha} &
\sHom_{\Ae}(\T, M_{3}) \ar[r] \ar[d]^-{-\smile \alpha} &
0 \\
0 \ar[r] &
\sHom_{\Ae}(\T, M_{1} \otimes_{A} N) \ar[r] &
\sHom_{\Ae}(\T, M_{2} \otimes_{A} N) \ar[r] &
\sHom_{\Ae}(\T, M_{3} \otimes_{A} N) \ar[r] &
0
}}\\[3pt]
The property (PII$_{1}$) follows from the naturality of connecting homomorphism.
\end{proof}
%
%

%
%
\begin{theo}[{\cite[Theorem 2.1]{Sanada92}}] \label{theo:2}
Any two cup products satisfying the three properties $({\rm PI})$--\,$({\rm PII}_{2})$ coincide up to isomorphism.
\end{theo}

Let us remark that a system of the properties $({\rm PI})$--\,$({\rm PII}_{2})$ may be originally seen in \cite{Kawa57}.

As a consequence of the two statements above, we have the following.

%
%
\begin{cor}
The cup product 
\begin{align*} 
\smile: \hHh^{*}(A, M) \otimes \hHh^{*}(A, N) \rightarrow \hHh^{*}(A, M \otimes_{A} N).
\end{align*} 
does not depend on the choice of a complete resolution and a diagonal approximation.
\end{cor}

We know that the cup product satisfies the following properties.

%
%
\begin{theo}[{\cite[Propositions 2.3 and 2.4]{Sanada92}}] \label{theo:9}
There exists one diagonal approximation associated with a certain complete resolution $\mathbf{X} \xrightarrow{\varepsilon} A$ such that
\begin{enumerate}
\item[(i)]
The induced cup product $\smile$ is associative, i.e.,
\[
(u_{1} \smile u_{2}) \smile u_{3} = u_{1} \smile ( u_{2} \smile u_{3} )
\]
for $u_{i} \in \hHh^{*}(A, M_{i})$ with an $A$-bimodule $M_{i}$.
\item[(ii)]
The induced cup product $\smile$ endows a graded vector space  
\[
\hHh^{\bullet}(A, A) := \bigoplus_{i \in \Z} \hHh^{i}(A, A) = \bigoplus_{i \in \Z} \Hh^{i}(\sHom_{\Ae}(\mathbf{X}, A))
\]
with a graded commutative algebra structure whose unit is the element represented by the augmentation $\varepsilon$.
\end{enumerate}
\end{theo}
%
%

Secondly, we deal with the cap product in an analogous way to the cup product.
%
%
\begin{prop} \label{prop:5}
The cap product $\frown$ satisfies the following three properties.
\begin{enumerate}
\item[(QI)]
    Let $M$ and $N$ be $A$-bimodules. Then there exists a commutative square
    \[
    \xymatrix{
    \hHh^{0}(A, M) \otimes \hHh_{0}(A, N) \ar[r]^{\frown} \ar[d]&
    \hHh_{0}(A, M \otimes_{A} N) \ar[d] \\
    M^{A}/N_{A}(M) \otimes {}_{N_{A}}N /I_{A}(N) \ar[r]   &
    {}_{N_{A}}(M \otimes_{A} N)/I_{A}(M \otimes_{A} N)
    }\]
    where the vertical morphisms are isomorphisms in $(\ref{eq:5})$ and the morphism in the bottom row is given by
    \[
    (m + N_{A}(M)) \otimes (n + I_{A}(N)) \mapsto m \otimes_{A} n + I_{A}(M \otimes_{A} N).
    \]
\item[(QII$_{1}$)]
    Let 
    \begin{align*}
        &0 \rightarrow M_{1} \rightarrow M_{2} \rightarrow M_{3} \rightarrow 0, \\
        &0 \rightarrow M_{1} \otimes_{A} N \rightarrow M_{2} \otimes_{A} N \rightarrow M_{3} \otimes_{A} N \rightarrow 0
    \end{align*}
    be exact sequences of $A$-bimodules. Then we have 
    \[
    \partial(\gamma \frown \xi) = (\partial \gamma) \frown \xi
    \]
    for all $\gamma \in \hHh^{r}(A, M_{3})$ and $\xi \in \hHh_{s}(A, N)$, 
    where $\partial$ denotes the connecting homomorphism.
\item[(QII$_{2}$)]
    Let 
    \begin{align*}
        &0 \rightarrow N_{1} \rightarrow N_{2} \rightarrow N_{3} \rightarrow 0, \\
        &0 \rightarrow M \otimes_{A} N_{1} \rightarrow M \otimes_{A} N_{2} \rightarrow M \otimes_{A} N_{3} \rightarrow 0
    \end{align*}
    be exact sequences of $A$-bimodules. Then we have 
    \[
    \partial(\gamma \frown \xi) = (-1)^{r}  \gamma \frown (\partial \xi )
    \]
    for all $\gamma \in \hHh^{r}(A, M)$ and $\xi \in \hHh_{s}(A, N_{3})$, 
    where $\partial$ denotes the connecting homomorphism.
\end{enumerate}
\end{prop} 
\begin{proof}
The proof of this proposition is similar to that of Proposition $\ref{prop:1}$.
\end{proof} 
%
%

%
%
\begin{theo} \label{theo:4}
There exists only one cap product satisfying the three properties $({\rm QI})$--\,$({\rm QII}_{2})$.
\end{theo} 
\begin{proof}
Let $\frown$ and $\frown^{\prime}$ be any two cap products, and let $r, s$ be arbitrary integers.
The property $(\textrm{QI})$ implies that $\frown$ coincides with $\frown^{\prime}$ in the case $r = s = 0$.
A dimension-shifting argument analogue to that in \cite[pp.\,78--79]{Sanada92} yields that $\frown$ agrees with $\frown^{\prime}$ for any $r, s \in \Z$.
\end{proof} 
%
%

%
%
\subsection{Composition products} 
In this subsection, we will show, as in Hochschild theory, that 
the cup product and the cap product on Tate-Hochschild (co)homology coincide with some composition products.
The following proposition is crucial.
%
%
\begin{prop} \label{prop:3}
Let $\T$ be an acyclic chain complex of $($not necessarily finitely generated$)$ projective $A$-bimodules, 
$\T^{\prime} \xrightarrow{\varepsilon^{\prime}} A$ a complete resolution and 
$M$ an $A$-bimodule. Then the following statements hold.
\begin{enumerate}
\item[(1)]
    A chain map $\varepsilon^{\prime} \otimes_{A} \id_{M} : \T^{\prime} \otimes_{A} M \rightarrow M$ induces a quasi-isomorphism
\begin{align*}
    \sHom_{\Ae}(\T, \varepsilon^{\prime} \otimes_{A} \id_{M}): 
    \sHom_{\Ae}(\T, \T^{\prime} \otimes_{A} M) \rightarrow  \sHom_{\Ae}(\T, M).
\end{align*}
\item[(2)]
    If each component of $\T$ is finitely generated, then $\varepsilon^{\prime} \otimes_{A} \id_{M}$ induces a quasi-isomorphism
    \[
    \id_{\T} \hotimes_{\Ae} (\varepsilon^{\prime} \otimes_{A} \id_{M}):
    \T \hotimes_{\Ae} (\T^{\prime} \otimes_{A} M) \rightarrow \T \otimes_{\Ae} M.
    \]
\end{enumerate}
In particular, if $\T$ is a compete resolution of $A$, then we have isomorphisms
\begin{align*}
&\hHh^{*}(A, M) \cong \Hh^{*}(\sHom_{\Ae}(\T, \T^{\prime} \otimes_{A} M)), \\[3pt]
&\hHh_{*}(A, M) \cong \Hh^{*}(\T \hotimes_{\Ae} (\T^{\prime} \otimes_{A} M).
\end{align*}
\end{prop}
\begin{proof}
First, we prove $(1)$. Since $\boldsymbol{\Sigma}^{n} \T$ is an acyclic complex of projective $A$-bimodules for all $n \in \Z$, it suffices to show 
that the induced morphism
$[\T, \T^{\prime} \otimes_{A} M]$ $\rightarrow$ $[\T, M]$
is an isomorphism.
Take a chain map $u: \T \rightarrow M$.
Since $\varepsilon^{\prime} \otimes_{A} \id_{M}$ is an epimorphism and $T_{0}$ is projective, 
there exists a morphism $\tau_{0}: T_{0} \rightarrow T^{\prime}_{0} \otimes_{A} M$ such that $(\varepsilon^{\prime} \otimes_{A} \id) \tau_{0} = u$.
Then we get
\begin{align*}
    (d_{0}^{\T^{\prime}} \otimes_{A} \id_{M}) \tau_{0} d^{\T}_{1}
    &=
    (\eta^{\prime} \otimes_{A} \id_{M})(\varepsilon^{\prime} \otimes_{A} \id_{M}) \tau_{0} d^{\T}_{1}\\
    &=
    (\eta^{\prime} \otimes_{A} \id_{M}) u d^{\T}_{1}
    =
    0.
\end{align*}
It follows from Lemma \ref{lem:4} that $\tau_{0}$ extends to a chain map $\tau: \T \rightarrow \T^{\prime} \otimes_{A} M$ such that $\sHom_{\Ae}(\T, \varepsilon^{\prime} \otimes_{A} \id_{M})(\tau) = u$.
On the other hand, assume that $\tau: \T \rightarrow \T^{\prime} \otimes_{A} M$ is a chain map such that 
$(\varepsilon^{\prime} \otimes_{A} \id)\tau: \T \rightarrow M$ is null-homotopic.
Then there exists a morphism $s: T_{-1} \rightarrow M$ such that $(\varepsilon^{\prime} \otimes_{A} \id)\tau_{0} = s d^{\T}_{0}$.
Since $\varepsilon^{\prime} \otimes_{A} \id$ is an epimorphism and $T_{-1}$ is projective, there exists a lifting 
$h_{-1}: T_{-1} \rightarrow T^{\prime}_{0} \otimes_{A} M$ of $s$ such that $(\varepsilon^{\prime} \otimes_{A} \id) h_{-1} = s$, and we have
\begin{align*}
    (d_{0}^{\T^{\prime}} \otimes_{A} \id) h_{-1} d^{\T}_{0}
    &=
    (\eta^{\prime} \otimes_{A} \id)(\varepsilon^{\prime} \otimes_{A} \id) h_{-1} d^{\T}_{0} \\
    &=
    (\eta^{\prime} \otimes_{A} \id) s d^{\T}_{0} \\
    &=
    (d_{0}^{\T^{\prime}} \otimes_{A} \id) \tau_{0}.
\end{align*}
As in the proof of Lemma $\ref{lem:7}$, we can construct $h_{0}: T_{0} \rightarrow T^{\prime}_{1} \otimes_{A} M$ such that 
$\tau_{0} = (d_{1}^{\T^{\prime}} \otimes_{A} \id) h_{0} +h_{-1} (d_{0}^{\T^{\prime}} \otimes_{A} \id)$.
Inductively, we obtain a family $\{ h_{i}\}_{i \geq -1}$ satisfying 
$\tau_{i} =  (d_{i+1}^{\T^{\prime}} \otimes_{A} \id) h_{i} +h_{i-1} (d_{i}^{\T^{\prime}} \otimes_{A} \id)$ for $i \geq 0$.
It follows from Lemma $\ref{lem:8}$ that $\{ h_{i}\}_{i \geq -1}$ extends to a null-homotopy of $\tau$.

In order to prove (2), take the $\Ae$-dual complex $\T^{\vee}$ of $\T$.
It is still an acyclic complex of finitely generated projective $A$-bimodules.
Thus, the second statement follows from the first statement and the fact that there exists an isomorphism 
\begin{align*}
    \T \hotimes_{\Ae} C \cong \sHom_{\Ae}(\T^{\vee}, C) 
\end{align*}
for any chain complex $C$ over $\Ae$.
\end{proof}
%
%

Let $\T$ be a complete resolution of $A$ and $M$ and $N$ $A$-bimodules. 
We consider the following two chain maps: the first is the composition map
\begin{align} \label{eq:9}
    \sHom_{\Ae}(\T, M) \otimes  \sHom_{\Ae}(\T, \T \otimes_{A} N) \rightarrow  \sHom_{\Ae}(\T, M \otimes_{A} N)
\end{align}
defined by 
\[
u \otimes v \mapsto (u \otimes_{A} \id_{N}) \circ v
\]
for any homogeneous element $u \otimes v \in \sHom_{\Ae}(\T, M) \otimes \sHom_{\Ae}(\T, \T \otimes_{A} N)$, and the second is the chain map
\begin{align} \label{eq:11}
    \sHom_{\Ae}(\T, M) \otimes  (\T \hotimes_{\Ae} (\T \otimes_{A} N)) \rightarrow   \T \otimes_{\Ae} (M \otimes_{A} N)
\end{align}
given by 
\[
u \otimes (x \otimes_{\Ae} (x^{\prime} \otimes_{A} n)) \mapsto (-1)^{|u||x|} x \otimes_{\Ae} (u(x^{\prime}) \otimes_{A} n)
\]
for any homogeneous element 
$u \otimes (x \otimes_{\Ae} (x^{\prime} \otimes_{A} n)) \in 
\sHom_{\Ae}(\T, M) \otimes  (\T \hotimes_{\Ae} (\T \otimes_{A}$ $N))$.
By Proposition $\ref{prop:3}$, these chain maps induce a well-defined operators, called \textit{composition products},
\begin{align} 
  \label{eq:7}  &\hHh^{r}(A, M) \otimes  \hHh^{s}(A, N) \rightarrow \hHh^{r+s}(A, M \otimes_{A} N), \\[3pt]
 \label{eq:12}   &\hHh^{r}(A, M) \otimes  \hHh_{s}(A, N) \rightarrow \hHh_{s-r}(A, M \otimes_{A} N).
\end{align}

%
%
\begin{theo} \label{theo:3}
The following statements hold.
\begin{enumerate}
\item[(1)]
    The composition product $(\ref{eq:7})$ agrees with the cup product $(\ref{eq:8})$.
\item[(2)]
    The composition product $(\ref{eq:12})$ agrees with the cap product $(\ref{eq:10})$.
\end{enumerate}
\end{theo}
\begin{proof}
First of all, we show $(1)$.
The composition product $(\ref{eq:7})$ is induced by the chain map $(\ref{eq:9})$ via the quasi-isomorphism $\alpha := \sHom_{\Ae}(\T, \varepsilon^{\prime} \otimes_{A} \id_{M})$.
We now construct a quasi-inverse of $\alpha$, i.e., a chain map 
\begin{align*}
\alpha^{\prime}: \sHom_{\Ae}(\T, N) \rightarrow \sHom_{\Ae}(\T, \T^{\prime} \otimes_{A} N)
\end{align*}
inducing the inverse of the isomorphism $\Hh(\alpha)$.
Let $\alpha^{\prime}$ be the chain map
\begin{align*}
\sHom_{\Ae}(\T, N) \rightarrow \sHom_{\Ae}(\T, \T^{\prime} \otimes_{A} N)
\end{align*}
determined by $u \mapsto (\id_{\T} \hotimes_{A} u) \widehat{\Delta}$,
where $\widehat{\Delta}: \T \rightarrow \T \hotimes_{A} \T$ is a diagonal approximation.
Then the composition $\alpha \alpha^{\prime}: \sHom_{\Ae}(\T, N) \rightarrow \sHom_{\Ae}(\T, N)$ is given by
\[
\alpha \alpha^{\prime} 
    = (\varepsilon \otimes_{A} N)(\id_{\T} \hotimes_{A} u) \widehat{\Delta} 
    = (\varepsilon \hotimes_{A} u) \widehat{\Delta} 
    = \varepsilon \smile u.
\]
Since $[\varepsilon] = 1 \in \hHh^{0}(A, A)$, we see that
$\alpha^{\prime}$ is a quasi-inverse of $\alpha$. 
We are now able to compare the composition product $(\ref{eq:7})$ with the cup product $(\ref{eq:8})$ as follows:
for any $u \in \sHom_{\Ae}(\T, M)$ and $v \in \sHom_{\Ae}(\T, N)$,
\begin{align*}
    u \otimes v 
    &\mapsto 
    u \otimes \alpha^{\prime}(v)\\
    &\mapsto 
    (u \hotimes_{A} \id)( \id \hotimes_{A} v) \widehat{\Delta} \\
    &=
    (u \hotimes_{A} v) \widehat{\Delta} 
    =
    u \smile v.
\end{align*}

Secondly, we prove (2). In view of the proof of $(1)$, we construct a weak-inverse $\beta^{\prime}$ of 
$\beta:= \id_{\T} \hotimes_{\Ae} (\varepsilon \otimes_{A} \id_{N}):
\T \hotimes_{\Ae} (\T \otimes_{A} N) \rightarrow  \T \otimes_{\Ae} N$.
Let $\beta^{\prime}$ be the composition
\begin{align*}
    \T \otimes_{\Ae} N \xrightarrow{ \widehat{\Delta} \otimes \id} (\T \hotimes_{A} \T) \otimes_{\Ae} N \cong  \T \hotimes_{\Ae} (\T \otimes_{\Ae} N ).
\end{align*}
Note that there exists a commutative square of chain complexes
\[\xymatrix@C=70pt{
(\T \hotimes_{A} \T) \otimes_{\Ae} N \ar[r]^{(\id \hotimes_{A} \varepsilon) \otimes_{\Ae} \id} \ar[d]_{\cong} &
(\T \otimes_{A} A) \otimes_{\Ae} N  \ar[d]^{\cong} \\
\T \hotimes_{A} (\T \otimes_{\Ae} N) \ar[r]^{\id \hotimes_{A} (\varepsilon \otimes_{\Ae} \id)}  &
\T \otimes_{A} (A \otimes_{\Ae} N )
}\]
Thus we see that the morphism 
\begin{align*}
    \beta \beta^{\prime}: \T \otimes_{\Ae} N \rightarrow \T \otimes_{\Ae} N  
\end{align*}
is of the form $ \sigma \otimes_{\Ae} \id_{N}$, where $\sigma:\T \rightarrow \T$ is an augmentation-preserving chain map and hence is homotopic to $\id_{\T}$.
It is easy to check that the two chain map $(\ref{eq:12})$ and $(\ref{eq:10})$ coincide via the weak-inverse $\beta^{\prime}$ on the chain level.
This completes the proof. 
\end{proof}
%
%

%
%
\begin{rem} \label{rem:2}
{\rm
In the proof of Theorem $\ref{theo:3}$, we can directly construct  the inverse of $\Hh(\alpha)$ when $N=A$ in the following way: 
let $n \in \Z$ and $g \in \sHom_{\Ae}(\T, A)^{n}$ a cocycle. 
Then there uniquely (up to homotopy) exists a cocycle $\overline{g} \in \sHom_{\Ae}(\T, \T)^{n}$. Indeed, since $g$ is a cocycle, there uniquely exists a morphism 
$g^{\prime}: \Coker d^{\boldsymbol{\Sigma}^{-n}\T}_{1} \rightarrow T_{-1}$ 
making the center square in the following diagram commute:
\[\xymatrix@C=31pt@R=20pt{
&
&
\Coker d^{\boldsymbol{\Sigma}^{-n}\T}_{1} \ar[rd] \ar[rdd]_(.6){g^{\prime}}&
&
 \\
\cdots \rightarrow 
T_{n+1} \ar[r]^-{d^{\boldsymbol{\Sigma}^{-n}\T}_{1}} & 
T_{n} \ar[rr]^-{d^{\boldsymbol{\Sigma}^{-n}\T}_{0}}  \ar@{->>}[ru] \ar[rdd]^(.35){g}& 
&
T_{n-1} \ar[r]^-{d^{\boldsymbol{\Sigma}^{-n}\T}_{-1}} & 
T_{n-2}  \rightarrow \cdots  \\ 
\cdots \rightarrow 
T_{1} \ar[r]^-{d_{1}}&
T_{0} \ar[rr]^-{d_{0}} \ar@{->>}[rd]&
&
T_{-1} \ar[r]^-{d_{-1}}&
T_{-2} \rightarrow \cdots\\
&
&
A \ar[ru]&
&
}\]
By Lemmas $\ref{lem:7}$ and $\ref{lem:8}$, there uniquely (up to homotopy) exists 
a lifting chain map $\overline{g}: \boldsymbol{\Sigma}^{-n}\T \rightarrow \T$ of $g$. 
Moreover, we can take $d^{\sHom_{\Ae}(\T, \T)}_{n}(\overline{g})$ 
as a lifting chain map of a coboundary $d^{\sHom_{\Ae}(\T, A)}_{n}(g)$.
Thus, we obtain the morphism
\begin{align*}
    \beta: \hHh^{n}(A, A) \rightarrow \Hh^{n}(\sHom_{\Ae}(\T, \T)) 
\end{align*}
given by $[g] \mapsto [\overline{g}]$.
This map is the inverse of $\Hh(\alpha)$. Indeed, we have 
\begin{align*}
    \Hh(\alpha) \beta([g]) = \Hh(\alpha) ([\overline{g}]) = [\alpha(\overline{g})] =[g]
\end{align*}
for all $[g] \in \hHh^{n}(A, A)$. 
Thus, we get $\beta = \Hh(\alpha^{\prime})$.
In conclusion, we can compute the cup product $\smile$ via chain maps when $N =A$.
}\end{rem}
%
%

As a corollary of Theorem \ref{theo:3}, we have the following property
with respect to the cup product and the cap product.

%
%
\begin{cor} \label{cor:3}
For all $\alpha \in \hHh^{*}(A, L), \beta \in \hHh^{*}(A, M)$ and $\omega \in \hHh_{*}(A, N)$, we have 
\[
(\alpha \smile \beta) \frown \omega = \alpha \frown (\beta \frown \omega).
\]
\end{cor}
\begin{proof}
In view of Theorem $\ref{theo:3}$, it suffices to show the statement for the elements represented by
$u \in \sHom_{\Ae}(\T, L)$, $v \in \sHom_{\Ae}(\T, \T \otimes_{A} M)$ and $ z= x \otimes_{\Ae} (y \otimes_{A} n )\in \T \hotimes_{\Ae} (\T \otimes_{A} N)$.
Then we have, on the chain level,
\begin{align*}
(u \smile v) \frown z 
    &=
    (-1)^{(|u|+|v|)|x|} x \otimes_{\Ae} (u \smile v)(y) \otimes_{A} n \\
    &=
    (-1)^{(|u|+|v|)|x|} x \otimes_{\Ae} (u \otimes_{A} \id) v(y) \otimes_{A} n, \\
u \frown (v \frown z)
    &= 
    u \frown ((-1)^{|v||x|} x \otimes_{\Ae} v(y) \otimes_{A} n) \\
    &= 
    (-1)^{|u||x|+|v||x|} x \otimes_{\Ae} (u \otimes_{A} \id) v(y) \otimes_{A} n. 
\end{align*}
This finishes the proof.
\end{proof}
%
%

%
%
\subsection{Comparison with singular Hochschild cohomology ring}
Our aim in this subsection is to prove our main theorem.
For this, we first recall the result of Eu and Schedler.
Let $f \in \underline{\Hom}_{\Ae}(\Omega_{\Ae}^{i}A, A)$ and 
$g \in \underline{\Hom}_{\Ae}(\Omega_{\Ae}^{j}A, A)$ with $i, j \in \Z$. 
We define a morphism 
$f \cup g$ by
\begin{align*}
    f \cup g := f  \Omega_{\Ae}^{i}(g)  
    \in \underline{\Hom}_{\Ae}(\Omega_{\Ae}^{i+j}A, A),
\end{align*}
where $ \Omega_{\Ae}^{i}(g) \in \underline{\Hom}_{\Ae}(\Omega_{\Ae}^{i+j}A, \Omega_{\Ae}^{i}A)$ is induced by an autoequivalence 
$\Omega_{\Ae}: \Ae\textrm{-}\underline{\mathrm{mod}} \rightarrow \Ae\textrm{-}\underline{\mathrm{mod}}$.
On the other hand, it follows from \cite[Proposition 2.1.8]{EuSched09} that there exists an isomorphism
\begin{align} \label{eq:13}
    \Omega_{\Ae}^{i}L \underline{\otimes}_{\Ae} N \cong L \underline{\otimes}_{\Ae}  \Omega_{\Ae}^{i}N
\end{align}
for any finitely generated $A$-bimodules $L$ and $N$ and $i \in \Z$. Using this isomorphism, we define 
\begin{align*}
    \cap: 
    \underline{\Hom}_{\Ae}(\Omega_{\Ae}^{i}A, A) \otimes (\Omega_{\Ae}^{j}A \underline{\otimes}_{\Ae} M)
    \rightarrow
    \Omega_{\Ae}^{j-i}A \underline{\otimes}_{\Ae} M
\end{align*}
to be the morphism making the following square commute:
\[\xymatrix{
\underline{\Hom}_{\Ae}(\Omega_{\Ae}^{i}A, A) \otimes (\Omega_{\Ae}^{j}A \underline{\otimes}_{\Ae} M) \ar[r]^-{\cap} \ar[d]_{\cong} &
\Omega_{\Ae}^{j-i}A \underline{\otimes}_{\Ae} M \\
\underline{\Hom}_{\Ae}(\Omega_{\Ae}^{i}A, A) \otimes (\Omega_{\Ae}^{i}A \underline{\otimes}_{\Ae} \Omega_{\Ae}^{j-i}M) \ar[r]^-{\phi} &
A \underline{\otimes}_{\Ae} \Omega_{\Ae}^{j-i}M \ar[u]_{\cong}
}\]
where the morphism $\phi$ is given by 
\[
[f] \otimes (a \otimes_{\Ae} m) \mapsto f(a) \otimes_{\Ae} m.
\]

%
%
\begin{theo}[{\cite[Theorem 2.1.15]{EuSched09}}] \label{theo:6}
We have the following statements.
\begin{enumerate}
\item 
    The graded vector space 
    \[
    \underline{\Hom}_{\Ae}(\Omega_{\Ae}^{\bullet}A, A) := \bigoplus_{i \in \Z} \underline{\Hom}_{\Ae}(\Omega_{\Ae}^{i}A, A)
    \] 
    equipped with $\cup$ forms a graded commutative algebra, which extends the cup product on 
    $\Hh^{\geq 1}(A, A)$.
    \item  
        The morphism $\cap$ extends the cap product between $\Hh^{*}(A, A)$ and $\Hh_{*}(A, M)$
        for a finitely generated $A$-bimodule $M$ and satisfies the relation
        \[
        (f \cup g) \cap z = f \cap (g \cap z)
        \]
        for any $z \in \Omega_{\Ae}^{*}A \underline{\otimes}_{\Ae} M, f \in \underline{\Hom}_{\Ae}(\Omega_{\Ae}^{*}A, A)$ and 
        $g \in \underline{\Hom}_{\Ae}(\Omega_{\Ae}^{*}A, A)$.
\end{enumerate}
\end{theo}
%
%

It follows from Proposition $\ref{prop:9}$ that there exist isomorphisms
\begin{align*}
    \hHh^{*}(A, M) \cong \underline{\Hom}_{\Ae}(\Omega_{\Ae}^{*}A, M) \ \  \mbox{and}\ \  
    \hHh_{*}(A, M) \cong \Omega_{\Ae}^{*}A \underline{\otimes}_{\Ae} M
\end{align*}
for a finitely generated $A$-bimodule $M$.
We will prove that $\smile$ and $\frown$ are equivalent to $\cup$ and $\cap$, respectively, via the isomorphisms above.
%
%
\begin{lemma} \label{lem:12}
We have the following statements.
\begin{enumerate}
\item 
    There exists an isomorphism
    \[
    \hHh^{\bullet}(A, A) \cong \underline{\Hom}_{\Ae}(\Omega_{\Ae}^{\bullet}A, A) 
    \] 
    of graded commutative algebras.
\item  
    For any $i, j \in \Z$, there exists a commutative diagram
    \[\xymatrix{
    \underline{\Hom}_{\Ae}(\Omega_{\Ae}^{i}A, A) \otimes (\Omega_{\Ae}^{j}A \underline{\otimes}_{\Ae} M) \ar[r]^-{\cap} \ar[d] &
    \Omega_{\Ae}^{j-i}A \underline{\otimes}_{\Ae} M \ar[d] \\
    \hHh^{i}(A, A) \otimes \hHh_{j}(A, M) \ar[r]^-{\frown} &
    \hHh_{j-i}(A, M)
    }\]
    where $M$ is a finitely generated $A$-bimodule and the vertical morphisms are given by the isomorphisms  as in Proposition \ref{prop:9}.
\end{enumerate}
\end{lemma}
\begin{proof}
In order to prove (1), we will show that the following diagram commutes for every $i, j \in \Z$:
\[\xymatrix{
\underline{\Hom}_{\Ae}(\Omega_{\Ae}^{i}A, A) \otimes \underline{\Hom}_{\Ae}(\Omega_{\Ae}^{j}A, A) \ar[r]^-{\cup} \ar[d]_-{\Phi_{i} \otimes \Phi_{j}} &
\underline{\Hom}_{\Ae}(\Omega_{\Ae}^{i+j}A, A) \ar[dd]_-{\Phi_{i+j}} \\
\hHh^{i}(A, A) \otimes \hHh^{j}(A, A)  \ar[d]_-{\id \otimes \beta}  &
\\
\hHh^{i}(A, A) \otimes \Hh^{j}(\sHom_{\Ae}(\T^{A}, \T^{A})) \ar[r] &
\hHh^{i+j}(A, M) \\
}\]
where $\T^{A} \xrightarrow{\varepsilon} A$ is a minimal complete resolution, the lower horizontal morphism is the composition product, 
the morphism $\Phi_{*}$ is the isomorphism appeared in the proof of Proposition $\ref{prop:9}$ and
the morphism $\beta$ is the isomorphism constructed in Remark $\ref{rem:2}$.
Recall that we decompose each differential $d_{i}^{\T^{A}}$ as $d_{i}^{\T^{A}} = \iota_{i} \pi_{i}$ 
with $\pi_{i}: T_{i}^{A} \rightarrow \Omega_{\Ae}^{i}A$ and 
$\iota_{i}: \Omega_{\Ae}^{i}A \rightarrow T_{i-1}^{A}$. 
Let $[f] \otimes [g] \in \underline{\Hom}_{\Ae}(\Omega_{\Ae}^{i}A, A) \otimes \underline{\Hom}_{\Ae}(\Omega_{\Ae}^{j}A, A)$ be arbitrary.
Recalling the definitions of $\beta$ and of the autoequivalence $\Omega_{\Ae}$ of $\Ae\textrm{-}\underline{\mathrm{mod}}$, 
we have the following commutative diagram with exact rows:
\[\xymatrix@C=55pt{
0 \ar[r] &
\Omega_{\Ae}^{l+j+1}A \ar[r]^-{(-1)^{j}\iota_{l+j+1}} \ar[d]_-{\Omega_{\Ae}^{l+1}(g)}&
T_{l+j}^{A} \ar[r]^-{\pi_{l+j}} \ar[d]_-{g_{l}}&
\Omega_{\Ae}^{l+j}A \ar[d]^-{\Omega_{\Ae}^{l}(g)} \ar[r]&
0\\
0 \ar[r] &
\Omega_{\Ae}^{l+1}A \ar[r]^-{\iota_{l+1}} &
T_{l}^{A} \ar[r]^-{\pi_{l}} &
\Omega_{\Ae}^{l}A \ar[r]& 
0
}\]
where $l \in \Z$ and the morphism $g_{l}$ is a $l$-th component of a lifting cahin map of $g \pi_{j}$.
Then we have
\begin{align*}
    [f] \otimes [g]
    &\stackrel{{\small \Phi_{i} \otimes \Phi_{j}}}{\longmapsto} [f  \pi_{i}] \otimes [g \pi_{j}] \\
    &\stackrel{\id \otimes \beta}{\longmapsto} [f  \pi_{i}] \otimes \beta([g \pi_{j}]) \\
    &\longmapsto [(f  \pi_{i}) g_{i}]
\end{align*}
and 
\begin{align*}
    [f] \otimes [g] 
    &\stackrel{\cup}{\longmapsto} [f \Omega_{\Ae}^{i}(g)]  \\
    &\stackrel{\Phi_{i+j}}{\longmapsto} [(f \Omega_{\Ae}^{i}(g)) \pi_{i+j}] 
    = [f  (\Omega_{\Ae}^{i}(g)  \pi_{i+j})] = [f (\pi_{i} g_{i})].
\end{align*}
For the second statement, let $\T^{M}$ be a minimal complete resolution of $M$. For any $i \in \Z$, there exist two exact sequences
\begin{align*}
&0 \rightarrow \Omega_{\Ae}^{i+1}M \xrightarrow{\iota_{i+1}} T_{i}^{M} \xrightarrow{\pi_{i}} \Omega_{\Ae}^{i+1}M \rightarrow 0, \\
&0 \rightarrow A \otimes_{A} \Omega_{\Ae}^{i+1}M 
\xrightarrow{ \id \otimes_{A} \iota_{i+1}} A \otimes_{A} T_{i}^{M} 
\xrightarrow{\id \otimes_{A} \pi_{i}} A \otimes_{A} \Omega_{\Ae}^{i+1}M 
\rightarrow 0.
\end{align*}
It follows from \cite[Lemma 2.7]{ChrisJorgen14} that $\hHh_{r}(A, T_{s}^{M}) =0$ for all $r, s \in \Z$.
Thus the property $({\rm QII}_{2})$ of the cap product $\frown$ in Proposition $\ref{prop:5}$ implies that there exists a commutative square
\[\xymatrix{
\hHh^{i}(A, A) \otimes \hHh_{i}(A, \Omega_{\Ae}^{j-i}M) \ar[r]^-{\frown} \ar[d]&
\hHh_{0}(A, \Omega_{\Ae}^{j-i}M) \ar[d]\\
\hHh^{i}(A, A) \otimes \hHh_{j}(A, M) \ar[r]^-{\frown} &
\hHh_{0}(A, \Omega_{\Ae}^{j-i}M)
}\]
where $i, j \in \Z$ and the two vertical morphisms are isomorphisms. 
In order to complete the proof of (2), it suffices to show that a square
\[\xymatrix{
\underline{\Hom}_{\Ae}(\Omega_{\Ae}^{i}A, A) \otimes (\Omega_{\Ae}^{i}A \underline{\otimes}_{\Ae} \Omega_{\Ae}^{j-i}M) 
\ar[r]^-{\phi} \ar[d]_-{\Phi_{i} \otimes \Psi_{i}}&
A \underline{\otimes}_{\Ae} \Omega_{\Ae}^{j-i}M \ar[d]_-{\Psi_{0}}\\
\hHh^{i}(A, A) \otimes \hHh_{i}(A, \Omega_{\Ae}^{j-i}M) \ar[r]^-{\frown} &
\hHh_{0}(A, \Omega_{\Ae}^{j-i}M) 
}\]
is commutative for all $i, j \in \Z$.
Let 
\[
[f] \otimes (x \otimes_{\Ae} m) \in 
\underline{\Hom}_{\Ae}(\Omega_{\Ae}^{i}A, A) \otimes (\Omega_{\Ae}^{i}A \underline{\otimes}_{\Ae} \Omega_{\Ae}^{j-i}M)
\]
be arbitrary. 
Since 
$\pi_{i} \otimes_{\Ae} \id_{\Omega_{\Ae}^{j-i}M}$
is an epimorphism, we have $x \otimes_{\Ae} m = (\pi_{i} \otimes_{\Ae} \id)(a \otimes_{\Ae} m)$ for some $a \in T_{i}^{A}$.
Let $\widehat{\Delta}: \T^{A} \rightarrow \T^{A} \hotimes_{A} \T^{A}$ be a diagonal approximation, and we denote
\[
\widehat{\Delta}(a) = (x_{i-l} \otimes_{A} y_{l})_{l \in \Z} \in (\T^{A} \hotimes_{A} \T^{A})_{i}.
\]
Then we have 
\begin{align*}
    [f] \otimes (x \otimes_{\Ae} m)
    &\stackrel{\Phi_{i} \otimes \Psi_{i}}{\longmapsto} [f \pi_{i}] \otimes [a \otimes_{\Ae} m] \\
    &\stackrel{\frown}{\longmapsto} [x_{0} \otimes_{\Ae} f\pi_{i}(y_{i})m].
\end{align*}
On the other hand, since $\varepsilon \smile f \pi_{i}$ is homotopic to $f \pi_{i}$, 
there exists a morphism $h: T_{i-1}^{A} \rightarrow A$ such that 
$(\varepsilon \smile f \pi_{i}) -f \pi_{i} = h d_{i}^{\T^{A}}$.
Recalling the definition of the cup product $\smile$ on $\hHh$, we have 
\begin{align*}
&(\varepsilon \otimes_{\Ae} \id)( x_{0} f \pi_{i}(y_{i}) \otimes_{\Ae} m) \\
=&\ 
\varepsilon(x_{0}) f \pi_{i}(y_{i}) \otimes_{\Ae} m \\
=&\ 
(\varepsilon \smile f \pi_{i})(a) \otimes_{\Ae} m \\
=&\ 
f \pi_{i}(a) \otimes_{\Ae} m -(h \otimes_{\Ae} \id)(d_{i}^{\T^{A}} \otimes_{\Ae} \id)(a \otimes_{\Ae} m)\\
=&\ 
f(x) \otimes_{\Ae} m.
\end{align*}
Therefore, we get
\begin{align*}
    [f] \otimes (x \otimes_{\Ae} m)
    &\stackrel{\cap}{\longmapsto} [f(x) \otimes_{\Ae} m] \\
    &\stackrel{\Psi_{0}}{\longmapsto} [ x_{0} f \pi_{i}(y_{i}) \otimes_{\Ae} m].
\end{align*}
This completes the proof.
\end{proof}
%
%

We are now able to prove our main theorem.
Wang in \cite{Wang18} introduced \textit{singular Hochschild cochain complex} 
$C_{\textrm{sg}}(A, A)$ for any algebra $A$ over a field $k$ and defined 
the cup product $\cup_{\rm sg}$ on $C_{\textrm{sg}}(A, A)$.
We now recall the definitions of $C_{\textrm{sg}}(A, A)$ and of $\cup_{\rm sg}$ in \cite{Wang18}. 
The singular Hochschild cochain complex $C_{\textrm{sg}}(A, A)$ of $A$ is defined by the inductive limit of the inductive system of Hochschild cochain complexes 
\begin{align*}
    C(A, A) \stackrel{\theta_{0}}{\hookrightarrow}
    C(A, \Omega_{\rm nc}(A)) \hookrightarrow \cdots \hookrightarrow
    C(A, \Omega_{\rm nc}^{p}(A)) \stackrel{\theta_{p}}{\hookrightarrow} 
    C(A, \Omega_{\rm nc}^{p+1}(A)) \hookrightarrow \cdots,
\end{align*}
where $C(A, \Omega_{\rm nc}^{p}(A)) := \sHom_{\Ae}(\mathrm{Bar}(A), \Omega_{\rm nc}^{p}(A))$ with
the $A$-bimodule  $\Omega_{\rm nc}^{p}(A) := A \otimes \overline{A}^{\otimes p}$  
concentrated in degree $p$ of which the left action is the multiplication of $A$ and the right action is defined by
\begin{align*}
    (a_{0} \otimes \overline{a}_{1, p}) a_{p+1}
    := 
    \sum_{i=0}^{p}(-1)^{p-i} 
    a_{0} \otimes \overline{a}_{1,i} \otimes \overline{a_{i} a_{i+1}} \otimes \overline{a}_{i+2, p+1}
\end{align*}
for $a_{p+1} \in A$ and $a_{0} \otimes \overline{a}_{1, p} \in \Omega_{\rm nc}^{p}(A)$,
and the morphism $\theta_{p}$ is defined as
\[
C(A, \Omega_{\rm nc}^{p}(A)) \rightarrow C(A, \Omega_{\rm nc}^{p+1}(A)); \quad
f \mapsto f \otimes \id_{\overline{A}}.
\]
Here we have used the canonical isomorphism 
\[
\Hom_{\Ae}(\mathrm{Bar}(A)_{i}, \Omega_{\rm nc}^{p}(A))
\cong 
\Hom_{k}(\overline{A}^{\otimes i}, \Omega_{\rm nc}^{p}(A)). 
\]
For $i \in \Z$, we denote  
\[
\HH_{\rm sg}^{i}(A, A) := \Hh^{i}(C_{\textrm{sg}}(A, A)).
\]

Moreover, for $m, n, p, q \in \Z$, the cup product 
\begin{align*}
    \cup_{\rm sg}: 
    C^{m-p}(A, \Omega_{\rm nc}^{p}(A)) \otimes C^{n-q}(A, \Omega_{\rm nc}^{q}(A))
    \rightarrow
    C^{m+n-p-q}(A, \Omega_{\rm nc}^{p+q}(A))
\end{align*} 
is defined by
\[
 f \otimes g \mapsto 
 \left( \mu \otimes \id_{\overline{A}}^{\otimes p+q} \right) 
 \left( \id_{A} \otimes f \otimes \id_{\overline{A}}^{\otimes m} \right)
 \left( g \otimes \id_{\overline{A}}^{\otimes q}\right),
\]
where $\mu: A \otimes A \rightarrow A$ is the multiplication.
Wang has proved the following two results.
%
%
\begin{prop}[{\cite[Proposition 4.2 and Corollary 4.2]{Wang18}}]
Under the same notation above, the singular Hochschild cochain complex $C_{{\rm sg}}(A, A)$ 
equipped with the cup product $\cup_{\rm sg}$ forms a differential graded associative algebra such that
the induced cohomology ring $\HH_{\rm sg}^{\bullet}(A, A)$ is graded commutative.
\end{prop}
%
%

Before the second Wang's result, we  recall the definition of the singularity categories.
Let $A$ be a (two-sided) Noetherian algebra $A$ over a field $k$, and let $\mathcal{D}^{\textrm{b}}(A)$ be the bounded derived category of finitely generated $A$-modules. Then the \textit{singularity category} $\mathcal{D}_{{\rm sg}}(A)$ of $A$ is defined to be the Verdier quotient 
$\mathcal{D}_{{\rm sg}}(A)
=
\mathcal{D}^{\textrm{b}}(A) / \mathcal{K}^{\textrm{b}}(\mathrm{proj} A)$,
where $\mathcal{K}^{\textrm{b}}(\mathrm{proj} A)$ is the bounded homotopy category of finitely generated projective $A$-modules.
%
%
\begin{prop}[{\cite[Proposition 4.7]{Wang18}}]
Let $A$ be a Noetherian algebra over a field $k$. Then there exists an isomorphism 
\[
\HH_{{\rm sg}}^{\bullet}(A, A) \rightarrow 
\bigoplus_{i \in \Z} \Hom_{\mathcal{D}_{{\rm sg}}(\Ae)}(A, \boldsymbol{\Sigma}^{i} A) 
\]
of graded commutative algebras $($of degree $0)$,
where the product on the right hand side is given by the Yoneda product.
\end{prop}
%
%

If $A$ is a finite dimensional Frobenius algebra, then \cite[Theorem 2.1]{Rick89} implies that
the canonical functor $F: \Ae$-$\underline{\mathrm{mod}} \rightarrow \mathcal{D}_{\textrm{sg}}(\Ae)$ is 
an equivalence of triangulated categories such that 
$F \circ \Omega_{\Ae} \simeq \boldsymbol{\Sigma}^{-1} \circ F$.
Thus we have an isomorphism 
\[
\underline{\Hom}_{\Ae}(\Omega_{\Ae}^{\bullet}A, A) 
\rightarrow 
\bigoplus_{i \in \Z} \Hom_{\mathcal{D}_{\textrm{sg}}(\Ae)}(A, \boldsymbol{\Sigma}^{i} A)
\]
of graded algebras.
Consequently, from
Lemma $\ref{lem:12} (1)$, we have the following result, which is our main theorem.
%
%
\begin{theo}  \label{theo:7}
Let $k$ be a field and $A$ a finite dimensional Frobenius $k$-algebra. Then there exists an isomorphism
\[
\hHh^{\bullet}(A, A) \cong \HH_{{\rm sg}}^{\bullet}(A, A)
\] 
as graded commutative algebras. 
\end{theo}
%
%

It is easily checked that Tate-Hochschild cohomology rings are derived invariants of finite dimensional Frobenius algebras.
Indeed, suppose that two finite dimensional Frobenius algebras $A$ and  $B$ are derived equivalent,
i.e., $\mathcal{D}^{{\rm b}}(A)$ is equivalent to $\mathcal{D}^{{\rm b}}(B)$ as triangulated categories. Then there exists an equivalence 
$F_{\rm sg}: \mathcal{D}_{{\rm sg}}(\Ae) \rightarrow \mathcal{D}_{{\rm sg}}(B^{\rm e})$ 
 of triangulated categories such that
$F_{\rm sg}(A) \cong B$ (see \cite[Section 6]{Zimme14Book} for instance). 
Then we see that the isomorphism
\[
\Hom_{\mathcal{D}_{\textrm{sg}}(\Ae)}(A, \boldsymbol{\Sigma}^{i} A) 
\rightarrow
\Hom_{\mathcal{D}_{\textrm{sg}}(B^{\rm e})}(B, \boldsymbol{\Sigma}^{i} B) 
\]
induced by $F_{\rm sg}$ commutes with the Yoneda products.
Consequently, Theorem \ref{theo:7} yields our claim.

%
%
\section{Duality theorems in Tate-Hochschild theory}
\label{sec:4}
Let $A$ be a finite dimensional Frobenius algebra over a field $k$.
In this section, we prove that the Tate-Hochschild duality
\begin{align*}
    \hHh^{n}(A, M) \cong \hHh_{-n-1}(A, {}_{1}M_{\nu^{-1}})
\end{align*}
appeared in Section $\ref{sec:2}$ is induced by the cap product for any integer $n$ and any $A$-bimodule $M$, and we prove that the cup product on Tate-Hochschild cohomology
extends the cup product and the cap product on Hochschild (co)homology.

Applying the duality above for $n=0$ and $M=A$, we have 
\[
\hHh^{0}(A, A) \cong \hHh_{-1}(A, {}_{1}A_{\nu^{-1}}).
\]
Then an element $\omega \in \hHh_{-1}(A, {}_{1}A_{\nu^{-1}})$ is called the \textit{fundamental class} of $A$ 
if the image under the isomorphism above of $\omega$ is equal to $1 \in \hHh^{0}(A, A)$.

%
%
\begin{theo} \label{theo:8}
The fundamental class $\omega \in \hHh_{-1}(A, {}_{1}A_{\nu^{-1}})$ 
induces an isomorphism
-- $\frown \omega: \hHh^{n}(A, M) \rightarrow \hHh_{-n-1}(A, {}_{1}M_{\nu^{-1}})$
for any $n \in \Z$ and any $A$-bimodule $M$.
\end{theo}
\begin{proof}
If $\T$ is a complete resolution of $A$, there exist isomorphisms
\begin{align*}
    \sHom_{\Ae}(\T, M) 
    \cong 
    \T^{\vee} \otimes_{\Ae} M 
    \cong 
    {}_{\nu^{-1}}(\T^{\vee})_{1} \otimes_{\Ae} {}_{1}M_{\nu^{-1}},
\end{align*}
where ${}_{\nu^{-1}}(\T^{\vee})_{1}$ is an acyclic chain complex of finitely generated projective $A$-bimodules.
Observe that it is the 1-shifted complex of some complete resolution of the bimodule ${}_{\nu^{-1}}A_{\nu^{-1}}$.
Thus, there exists a complete resolution $\T^{\prime}$ of $A$ such that $\boldsymbol{\Sigma} \T^{\prime} \cong {}_{\nu^{-1}}(\T^{\vee})_{1}$,
so that we have isomorphisms
\begin{align*} 
    \hHh^{n}(A, M) 
    \cong 
    \Hh_{-n}(\boldsymbol{\Sigma} \T^{\prime} \otimes_{\Ae} {}_{1}M_{\nu^{-1}}) 
    \cong 
    \hHh_{-n-1}(A, {}_{1}M_{\nu^{-1}}).
\end{align*}
Clearly, the composite is natural in $M$ and compatible with long exact sequences in the following sense: for any short exact sequence of $A$-bimodules
\[
0 \rightarrow L \rightarrow M \rightarrow N \rightarrow 0,
\]
there exists a commutative diagram with long exact sequences
\[\xymatrix@C=15pt@R=25pt{
\cdots \ar[r]&
\hHh^{n}(A, M) \ar[r] \ar[d]_-{\cong} &
\hHh^{n}(A, N) \ar[r] \ar[d]_-{\cong} &
\hHh^{n+1}(A, L) \ar[r] \ar[d]_-{\cong}  &
\cdots\\
\cdots \ar[r]&
\hHh_{-n-1}(A, {}_{1}M_{\nu^{-1}}) \ar[r]&
\hHh_{-n-1}(A, {}_{1}N_{\nu^{-1}}) \ar[r]&
\hHh_{-n-2}(A, {}_{1}L_{\nu^{-1}}) \ar[r] &
\cdots
}\]  

Let $\varphi^{M}_{n}$ denote the isomorphism $\hHh^{n}(A, M) \rightarrow \hHh_{-n-1}(A, {}_{1}M_{\nu^{-1}})$.
We claim that $\varphi^{M}_{0}$ is given by the cap product with $\omega \in \hHH_{-1}(A, {}_{1}A_{\nu^{-1}})$.
For any $u \in \hHh^{0}(A, M)$, there exists a morphism $f : A \rightarrow M$ of $A$-bimodules such that $u$ is represented by $f$ and such that
$f$ induces morphisms
\begin{align*}
    \hHh^{0}(A, A) \rightarrow \hHh^{0}(A, M) \quad \mbox{and} \quad
    \hHh_{-1}(A, {}_{1}A_{\nu^{-1}} ) \rightarrow \hHh_{-1}(A, {}_{1}M_{\nu^{-1}})
\end{align*}
given by $v \mapsto u \smile v$ and by $w \mapsto u \frown w$, respectively.
Hence the naturality of $\varphi^{*}_{0}$ implies that 
we have 
\begin{align*}
    \varphi^{M}_{0}([u]) = \varphi^{M}_{0}([u] \smile 1) = [u] \smile \varphi^{A}_{0}(1) = [u] \frown \omega.
\end{align*}
A dimension-shifting argument shows that $\varphi^{M}_{*}$ coincides with -- $\frown \omega$ in all degrees.
 \end{proof}
%
%

It follows from Corollary $\ref{cor:3}$ and Theorem $\ref{theo:8}$ that the cup product is equivalent to the cap product in the following sense. 
%
%
\begin{cor} \label{cor:4}
For any $A$-bimodules $M, N$ and $r, s \in \Z$, there exists a commutative diagram 
\[\xymatrix@C=35pt@R=30pt{
\hHh^{r}(A, M) \otimes \hHh^{s}(A, N) \ar[r]^-{\smile} \ar[d]_{\id \otimes (- \frown \,\omega)}&
\hHh^{r+s}(A, M \otimes_{A} N) \ar[d]^{- \frown \,\omega}\\
\hHh^{r}(A, M) \otimes \hHH_{-s-1}(A, {}_{1}N_{\nu^{-1}}) \ar[r]^-{ \frown } &
\hHH_{-r-s-1}(A, {}_{1}(M \otimes_{A} N)_{\nu^{-1}})
}\]  
where the two vertical morphisms are isomorphisms.
\end{cor}
%
%

Let $\nu \in \Aut(A)$ be the Nakayama automorphism of $A$.
It is known that there exist two Tate-Hochschild duality results for Frobenius algebras:
\begin{align*}
\hHh_{i}(A, A) \cong D(\hHh_{-i-1}(A, A)) \mbox{ and } \hHh^{i}(A, A) \cong D(\hHh^{-i-1}(A, {}_{1}A_{\nu^{2}}))
\end{align*}
for any $i \in \Z$, where the first is proved by Eu and Schedler in \cite{EuSched09} and 
the second is proved by Bergh and Jorgensen in \cite{BerghJorgen13}.
We will give another proof of the two Tate-Hochschild duality results.
%
%
\begin{cor} \label{cor:7}
Let $\nu$ be the Nakayama automorphism of $A$. Then there exist two isomorphisms
\begin{align*}
\hHh_{i}(A, A) \cong D(\hHh_{-i-1}(A, A)) \mbox{ and } \hHh^{i}(A, A) \cong D(\hHh^{-i-1}(A, {}_{1}A_{\nu^{2}}))
\end{align*}
for all $i \in \Z$.
\end{cor}
\begin{proof}
Let $\T$ be a complete resolution of $A$ and $M$ a finitely generated $A$-bimodule. The adjointness of $\sHom$ and $\otimes$ implies that
there exists an isomorphism
\begin{align*}
    \sHom_{\Ae}(\T, D(M)) \cong \D (\T \otimes_{\Ae} M).
\end{align*}
Thus we have for any $i \in \Z$
\begin{align*}
    \hHh^{i}(A, D(M)) \cong \Hh^{i}(\D (\T \otimes_{\Ae} M)) \cong D(\hHh_{i}(A, M)).
\end{align*}
If $M = D(A)$, then we get
\begin{align*}
    \hHh^{i}(A, A) \cong D(\hHh_{i}(A, D(A))) \cong D(\hHh_{i}(A, {}_{1}A_{\nu}))\cong D(\hHh^{-i-1}(A, {}_{1}A_{\nu^{2}})),
\end{align*}
where the second isomorphism is induced by the isomorphism $D(A) \cong {}_{1}A_{\nu}$ of $A$-bimodules, and
the third isomorphism is induced by the isomorphism of Theorem $\ref{theo:8}$. 
Similarly, if $M =A$, then we obtain isomorphisms
\begin{align*}
    D(\hHh_{i}(A, A)) \cong \hHh^{i}(A, D(A)) \cong \hHh^{i}(A, {}_{1}A_{\nu}) \cong \hHh_{-i-1}(A, A).
\end{align*}

\end{proof}
%
%

Our next and last aim is to show 
that the cup product on Tate-Hochschild cohomology can be considered as
an extension of the cup product and the cap product on Hochschild (co)homology.
%
%
\begin{prop} \label{prop:7}
Let $M, N$ be $A$-bimodules and $r, s \in \Z$. 
We denote by $\widehat{\smile}$ the cup product on $\hHh$ and by $\widehat{\frown}$ the cap product on $\hHh$. 
Then the following two statements hold.
\begin{enumerate}
\item[(1)] 
    There exists a commutative square
    \[\xymatrix{
    \Hh^{r}(A, M) \otimes \Hh^{s}(A, N) \ar[r]^-{\smile} \ar[d] &
    \Hh^{r+s}(A, M \otimes_{A} N) \ar[d]  \\
    \hHh^{r}(A, M) \otimes \hHh^{s}(A, N) \ar[r]^-{\widehat{\smile}} &
    \hHh^{r+s}(A, M \otimes_{A} N)
    }\]
\item[(2)]
    There exist three commutative squares
    \begin{enumerate}
        \item[(a)] the case $r=0, s = 0$
           \[\xymatrix{
            \Hh^{0}(A, M) \otimes \hHh_{0}(A, N) \ar[r]^-{\frown} \ar[d] &
            \hHh_{0}(A, M \otimes_{A} N) \ar@{=}[d] \\
            \hHh^{0}(A, M) \otimes \hHh_{0}(A, N) \ar[r]^-{\widehat{\frown}} &
            \hHh_{0}(A, M \otimes_{A} N)
            }\]
        \item[(b)] the case $r=0, s > 0$
            \[\xymatrix{
            \Hh^{0}(A, M) \otimes \Hh_{s}(A, N) \ar[r]^-{\frown} \ar[d] &
            \Hh_{s}(A, M \otimes_{A} N) \ar@{=}[d] \\
            \hHh^{0}(A, M) \otimes \hHh_{s}(A, N) \ar[r]^-{\widehat{\frown}} &
            \hHh_{s}(A, M \otimes_{A} N)
            }\]
        \item[(c)] the case $r>0, s>0 $ with $ s \geq r$
            \[\xymatrix{
            \Hh^{r}(A, M) \otimes \Hh_{s}(A, N) \ar[r]^-{\frown} \ar@{=}[d] &
            \Hh_{s-r}(A, M \otimes_{A} N)  \\
            \hHh^{r}(A, M) \otimes \hHh_{s}(A, N) \ar[r]^-{\widehat{\frown}} &
            \hHh_{s-r}(A, M \otimes_{A} N) \ar[u]
            }\]
    \end{enumerate}
\end{enumerate}
In all of the four diagrams above, the vertical morphisms consist of the morphisms constructed after Lemma $\ref{lem:2}$.
\end{prop} 
\begin{proof}
Let $\T$ be a complete resolution of $A$, and we denote by $\P$ 
the truncation $\T_{\geq 0}$.
Let
$\Delta: \P \rightarrow \P \otimes_{A} \P$ be a diagonal approximation associated with a diagonal approximation 
$\widehat{\Delta}: \T \rightarrow \T \hotimes_{A} \T$, which has been constructed after
Theorem $\ref{theo:1}$.

Since the chain map $\Delta$ consists of non-negative components of $\widehat{\Delta} = \prod (\prod \widehat{\Delta}^{(n)}_{p})$, 
we have a commutative square
\[\xymatrix{
\sHom_{\Ae}(\P, M) \otimes \sHom_{\Ae}(\P, N) \ar[r]^-{\smile} 
\ar[d]_-{\sHom_{\Ae}(\vartheta, M) \otimes \sHom_{\Ae}(\vartheta, N)}  &
\sHom_{\Ae}(\P, M \otimes_{A} N) 
\ar[d]^-{\sHom_{\Ae}(\vartheta, M \otimes_{A} N)} \\
\sHom_{\Ae}(\T, M) \otimes \sHom_{\Ae}(\T, N) \ar[r]^-{\widehat{\smile}} &
\sHom_{\Ae}(\T, M \otimes_{A} N)
}\]
where $\vartheta$ is the canonical chain map $ \T \rightarrow \P$. 
Thus, we get the commutative square in $(1)$.

Recalling that $\hHh_{0}(A, N) \leq \Hh_{0}(A, N)$, we see that $u \frown z \in \hHh_{0}(A, M \otimes_{A} N)$ 
for $u \in \Hh^{0}(A, M)$ and $z \in \hHh_{0}(A, N)$.
Thus we obtain the commutative square in $($a$)$ of $(2)$.

For any $r \geq 0$ and $s >0$ with $s \geq r$, we have a commutative square
\[\xymatrix{
\sHom_{\Ae}(\T, M)^{r} \otimes (\T \otimes_{\Ae} M)_{s}  \ar[r]^-{ \frown } \ar[rd]_-{ \widehat{\frown} } &
 (\P \otimes_{\Ae} (M \otimes_{A} N))_{s-r} \\
&
(\T \otimes_{\Ae} (M \otimes_{A} N))_{s-r} \ar[u]_-{\vartheta \otimes_{\Ae}\id_{M \otimes_{A} N}}
}\]
compatible with the differentials.
Thus we have the remaining commutative squares of $(2)$.
\end{proof}
%
%

%
%
\begin{cor} \label{cor:5}
Let $r \geq 0$ and $s \geq 1$ be such that $r-s \leq -1$. Then there exists a commutative diagram
\[\xymatrix{ 
\Hh^{r}(A, M) \otimes \hHh_{s-1}(A, {}_{1}N_{\nu^{-1}}) \ar[r]^-{\frown}  \ar[d]&
\hHh_{s-r-1}(A, {}_{1}(M \otimes_{A} N)_{\nu^{-1}}) \ar[d] \\
\hHh^{r}(A, M) \otimes \hHh^{-s}(A, N) \ar[r]^-{\widehat{\smile}}  &
\hHh^{r-s}(A, M \otimes_{A} N)
}\]
where the vertical morphism on the right hand side is always an isomorphism, 
and the vertical morphism on the left hand side is an isomorphism if $r \not = 0$ and an epimorphism otherwise.
\end{cor}
\begin{proof}
Corollary $\ref{cor:4}$ and Proposition $\ref{prop:7}$ imply 
that there exists a commutative diagram
\[\xymatrix{
\hHh^{r}(A, M) \otimes \hHh^{-s}(A, N) \ar[r]^-{\widehat{\smile}} \ar[d]_-{\id \otimes(- \widehat{\frown}\,\omega)}  &
\hHh^{r-s}(A, M \otimes_{A} N) \ar[d]^-{- \widehat{\frown}\, \omega} \\
\hHh^{r}(A, M) \otimes \hHh_{s-1}(A, {}_{1}N_{\nu^{-1}}) \ar[r]^-{\widehat{\frown}}  &
\hHh^{r-s}(A, {}_{1}(M \otimes_{A} N)_{\nu^{-1}}) \\
\Hh^{r}(A, M) \otimes \hHh_{s-1}(A, {}_{1}N_{\nu^{-1}}) \ar[r]^-{\frown}  \ar[u]&
\hHh_{s-r-1}(A, {}_{1}(M \otimes_{A} N)_{\nu^{-1}}) \ar[u]
}\]
This completes the proof.
\end{proof}
%
%

\section*{Acknowledgments} 
The author would like to thank his PhD supervisor Professor Katsunori Sanada 
for giving such an interesting topic and 
for valuable comments and suggestions for the development of the paper.
The author also would like to thank Professor Tomohiro Itagaki 
for helpful discussions and comments on the paper and warm encouragement.

%
%

\bibliographystyle{alpha}
\bibliography{ref}

\end{document}